\documentclass[11pt, letterpaper]{amsart}
\usepackage{amssymb,latexsym, xcolor}
\usepackage{graphicx, float}
\usepackage{amsthm,amssymb}
\usepackage{amsmath}
\usepackage{mathrsfs}
\usepackage[dvips, hcentering,
includehead,width=14.2cm,
top=0.5cm,
height=21.6cm,
paperheight=23.8cm,paperwidth=17cm
]{geometry}
\usepackage{biblatex} 
\usepackage{hyperref}
\usepackage{enumitem}

\usepackage{tikz, tikz-cd}
\usetikzlibrary{external}
 \AtBeginEnvironment{tikzcd}{\tikzexternaldisable}
 \AtEndEnvironment{tikzcd}{\tikzexternalenable}

\newcommand{\cS}[1]{{\noindent\textsf{\color{purple}$\blacksquare$~#1~$\blacksquare$}}}
\newcommand{\hide}[1]{}

\newtheorem{theorem}{Theorem}[section]
\newtheorem{proposition}[theorem]{Proposition}
\newtheorem{conjecture}[theorem]{Conjecture}
\newtheorem{corollary}[theorem]{Corollary}
\newtheorem{lemma}[theorem]{Lemma}

\theoremstyle{definition}

\newtheorem{remark}[theorem]{Remark}

\newtheorem{example}[theorem]{Example}
\newtheorem{definition}[theorem]{Definition}

\numberwithin{equation}{section}

\newcommand{\mutation}[1]{\stackrel{#1}{\longleftrightarrow}}

\newenvironment{blue}{\relax\color{blue}}{\hspace*{.5ex}\relax}
\newcommand{\beb}{\begin{blue}}
\newcommand{\eb}{\end{blue}}

\DeclareFontFamily{U}{mathx}{\hyphenchar\font45}
\DeclareFontShape{U}{mathx}{m}{n}{
      <5> <6> <7> <8> <9> <10>
      <10.95> <12> <14.4> <17.28> <20.74> <24.88>
      mathx10
      }{}
\DeclareSymbolFont{mathx}{U}{mathx}{m}{n}
\DeclareFontSubstitution{U}{mathx}{m}{n}
\DeclareMathAccent{\widecheck}{0}{mathx}{"71}

\setenumerate{leftmargin=7.5mm}
\setitemize{leftmargin=5.5mm}

\begin{document}

\title{Mutation cycles from reddening sequences}

\author[T. J. Ervin]{Tucker J. Ervin}
\address{Department of Mathematics, University of Alabama,
	Tuscaloosa, AL 35487, U.S.A.}
\email{tjervin@crimson.ua.edu}

\author{Scott Neville}
\address{Department of Mathematics, University of Michigan, Ann Arbor, MI 48109, USA}
\email{nevilles@umich.edu}

\date{April 8, 2025.}

\thanks{Partially supported by NSF grants DMS-1840234, DMS-2054231 
and DMS-2348501 (S.~N.).}



\begin{abstract}
Given two quivers, each with a reddening sequence, we show how to construct a plethora of mutation cycles.
We give several examples, including a generalization of the construction of long mutation cycles in earlier work by the second author.
We also give new results on the reddening sequences of certain mutation-acyclic quivers and forks, classifying them in some cases.
\end{abstract}
\maketitle

\section{Introduction}

A quiver is a directed graph with no directed $1$ or $2$ cycles (but parallel arrows are allowed).
Mutations are involutions which transform a quiver.
The construction of cluster algebras is founded on the combinatorics of quivers and their mutations~\cite{MR1887642}. 
This paper relates two seemingly unrelated sequences of mutations: reddening sequences and mutation cycles.
A reddening sequence is a sequence of mutations that reverses all the arrows added in the ``framing" of a given quiver.
A mutation cycle is a (nontrivial) sequence of mutations which transform a given quiver into itself.

Our main results, Propositions~\ref{prop:extensions with equality give cycles} and~\ref{prop:extensions without equality give cycles}, are that one can take any two quivers $H,T$ with reddening sequences $\mathbf M_H,\mathbf M_T$ respectively, and construct a new quiver which lies on a mutation cycle. 
The construction is straightforward: 
take the quiver whose vertices are the disjoint union of those in $H$ and $T$;
add any number of arrows from vertices in $T$ to vertices in $H$.
(This is called a triangular extension.) 
The resulting quiver will be preserved, up to isomorphism, by mutating at the concatenation of the reddening sequences $\mathbf M_T \mathbf M_H$. 

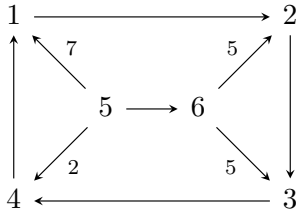
\begin{figure}[ht]
    \centering
\begin{tikzcd}[arrows={-stealth}, sep=1.8em]
1 & & & 2 \\
& 5 & 6 & \\
4 & & & 3
\arrow[from=1-1, to=1-4]
\arrow[from=1-4, to=3-4]
\arrow[from=3-4, to=3-1]
\arrow[from=3-1, to=1-1]
\arrow[from=2-2, to=2-3]
\arrow[to=1-1, from=2-2, "7"']
\arrow[to=3-1, from=2-2, "2"]
\arrow[to=1-4, from=2-3, "5"]
\arrow[to=3-4, from=2-3, "5"']
\end{tikzcd}
    \caption{A triangular extension of quivers of finite types $A_2$ (with vertices $\{5,6\}$) and $D_4$ (with vertices $\{1,2,3,4\}$). Mutating at vertices $ 5,6, 1, 2, 1, 3, 2, 4, 2, 1$ gives a mutation cycle.}
    \label{fig:example extension}
\end{figure}

Our construction is essentially immediate from existing theorems in the literature, but, to the best of our knowledge, it has not been noticed previously. 
We precisely describe the mutation cycle, and show that many of the mutation cycles constructed this way are simple 
(Theorem~\ref{thm:distinguishing minimal}).

Mutation cycles are useful for a few reasons. 
First, mutation cycles give nontrivial automorphisms in their cluster algebras; indeed, they are a necessary condition for the cluster automorphism group~\cite{AssemShifflerShramchenkoClusterAuts} (or cluster modular group~\cite{FockGoncharovEnsembles}) to be non-trivial.
Second, understanding the possible mutation cycles can help us design and test algorithms for mutation equivalence. 
Detecting the mutation equivalence of quivers is a major open problem. 
Third, one can use mutation cycles to create discrete time dynamical systems (e.g.\ by specializing cluster variables and iterating the cluster automorphism from the cycle, or by embedding the quivers into larger quivers~\cite{MachacekOvenhouseDynamicalMutation}).

Reddening sequences have also found several applications \cite{KellerGreenSurvey}.
For example, they give precise formulas for Donaldson-Thomas invariants \cite{kellerDT, KontsevichSoibelmanDT}, and, when a quiver has multiple reddening sequences, they provide quantum dilogarithm identities.
The automorphisms arising from reddening sequences (called twists) are often useful.

All of the quivers on the mutation cycles we construct have reddening sequences of their own, and so may be used to construct yet more mutation cycles. 
Many of these quivers are locally-acyclic (in fact, Banff), and so in particular their upper cluster algebra agrees with their cluster algebra \cite{muller_locally_2013}.
This construction includes the constructions in both Theorem~1.1 and Example~10.3 of \cite{LMC}. 
Unlike the cycles constructed there, some of these mutation cycles may be `paved' by shorter mutation cycles.
Also, these mutation cycles are often not unique in the mutation class - indeed, if one of the quivers used in the triangular extension has multiple reddening sequences then the triangular extension will lie on multiple mutation cycles.
Our results also give a plethora of examples of mutation cycles, with many free parameters, and without needing an explicit description of all the quivers on the mutation cycles.

Motivated by classifying mutation cycles, we also investigate properties of reddening sequences. 
We show that any reddening sequence must pass through the forkless part of the mutation graph, and is always conjugate to a reddening sequence entirely in the forkless part. 
Using these results, we classify all of the reddening sequences for abundant-acyclic quivers and keys.
These classifications also explicitly identify the isomorphism associated to the reddening sequences.

\smallskip

\noindent
\textbf{Paper organization.}
In Section~\ref{sec:prelim} we give precise definitions for quivers, mutation, mutation cycles, and reddening sequences. We also state the key results from the literature which we will use (in particular Corollaries~\ref{cor:reddening gives iso}~\ref{cor:extensions give iso cycles} and Theorems~\ref{thm:muller redd}~\ref{thm:extension red seq}).
In Section~\ref{sec:red makes cycles} we will state and prove our main results for mutation cycles, Propositions~\ref{prop:extensions with equality give cycles}~\ref{prop:extensions without equality give cycles} and Theorem~\ref{thm:distinguishing minimal}, as well as our structural results for reddening sequences.
We also state a conjecture that all reddening sequences with nontrivial associated permutations are due to some embedding of the finite type quiver $A_2$.
In Section~\ref{sec:examples} we provide many examples of reddening sequences, as well as references and brief descriptions to many more examples and constructions.
We then illustrate our results by combining these examples to give several new mutation cycles.
Also, in Example~\ref{eg:fordy-marsh not extension} we give a new example of a long mutation cycle which is not the result of a triangular extension.

\newpage
\section{Preliminaries and prior work}
\label{sec:prelim}

We briefly review the basics of quiver combinatorics, with a focus on reddening sequences and $C$-matrices.
More detail can be found in any of \cite{MR4695532, LMC, fomin2021introduction, MR3155783}. 

\begin{definition}
A \emph{quiver} $Q$ is a finite directed graph with multiple edges (or arrows) allowed, but no directed $2$-cycles or loops. 
A \emph{subquiver} of $Q$ is a full subgraph (i.e., we delete a subset of vertices and all arrows incident to the deleted vertices).

All of our quivers have labeled vertices. 
Two quivers are isomorphic if they agree up to a change of vertex labels. 
They are equal if they are equal as labeled graphs. 

We use $|Q|$ to denote the number of vertices in a quiver $Q$. This number is also called the \emph{rank} of the quiver.
\end{definition}

\begin{definition}
Fix a vertex $i$ in $Q$. 
To mutate $Q$ at $i$, apply the following operations:
\begin{enumerate}
\item for each directed path of length $2$ through $i$, $u \rightarrow i \rightarrow v$, add a new arrow $u \rightarrow v$;
\item reverse all arrows incident to $i$;
\item delete a maximal collection of oriented $2$-cycles. 
\end{enumerate}
We call the resulting quiver $\mu_i(Q)$. 
The \emph{mutation class} $\big [Q \big ]$ of a quiver $Q$ is the set of all quivers which can be obtained from $Q$ by a sequence of mutations.
\end{definition}

\begin{figure}[ht]
    \centering
    \begin{tikzcd}[every arrow/.append style = {-{Stealth}}]
        & 2 \\
        1 & & 3 \\
        \arrow[from=2-1,to=1-2, "4"]
        \arrow[from=1-2,to=2-3, "4"]
        \arrow[from=2-3,to=2-1, "4"]
    \end{tikzcd}
    \text{ and }
    \begin{tikzcd}[every arrow/.append style = {-{Stealth}}]
        & 2 \\
        1 & & 3 \\
        \arrow[from=2-3,to=1-2, "4"']
        \arrow[from=1-2,to=2-1, "4"']
        \arrow[from=2-1,to=2-3, "4"']
    \end{tikzcd}
    \caption{Two isomorphic but not equal quivers. An isomorphism is given by exchanging 1 and 2.}
    \label{fig-non-equal-example}
\end{figure}
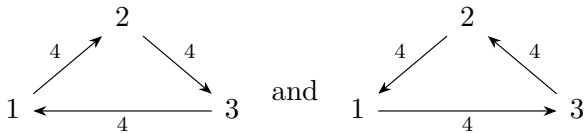

We will use the following properties of quiver mutation.

\begin{proposition}
    Let $Q$ be a quiver with a vertex $i$.
    \begin{itemize}
        \item We have $Q = \mu_i \circ \mu_i(Q)$; mutation is an involution.
        \item If $S$ is a subquiver of $Q$ which contains $i$, then $\mu_i(S)$ is a subquiver of $\mu_i(Q)$.
        \item If $i$ is a sink or source vertex in $Q$, then $\mu_i(Q)$ only reverses the arrows incident to $i$ (we do not add any new arrows, there are no oriented $2$-cycles to delete).
    \end{itemize}
\end{proposition}

\begin{definition}
A mutation sequence $\mathbf{i} = i_1, \ldots, i_m$ of size $|\mathbf{i}| = m$ is a sequence of vertices (of some quiver~$Q$). 
We say a mutation sequence $\mathbf{i}$ is \emph{reduced} if each consecutive pair of vertices are distinct (i.e.\ $i_k \neq i_{k \pm 1}$ for $1 < k < m$).
The \emph{reduction} of a mutation sequence $\mathbf i$ is the reduced subsequence remaining after repeatedly canceling any adjacent duplicates $i, i$ in $\mathbf i$. 
To mutate $Q$ at a mutation sequence, mutate at each vertex in $\mathbf{i}$ in order:
\[
\mu_{\mathbf i}(Q) := \mu_{i_m} \circ \mu_{i_{m-1}} \circ \cdots \circ \mu_{i_1}(Q).
\]

We will denote the reversed mutation sequence by~$\mathbf i^{-1} = i_m, \ldots, i_1$.
If we have two sequences of vertices $\mathbf i = i_1, \ldots, i_m, \mathbf j = j_1, \ldots, j_\ell$, we denote the concatenated sequence of vertices $\mathbf i \mathbf j = i_1, \ldots, i_m, j_1, \ldots, j_\ell$.

Given a quiver $Q$ and a mutation sequence $\mathbf{i}$ of vertices in $Q$, we use $\big [ Q \big ]_{\mathbf{i}}$ to denote the list of quivers $Q, \mu_{i_1}(Q), \ldots, \mu_{i_m} \circ \mu_{i_{m-1}} \circ \cdots \circ \mu_{i_1}(Q)$. 
\end{definition}

\begin{definition}
\label{def:mugraph}
The \emph{mutation graph} of a mutation class $\big [ Q \big ]$ has a vertex for each quiver in $\big [ Q \big ]$ and an (undirected) edge $R \mutation{i} S$ labeled $i$ between $R$ and $S$ whenever $\mu_i(R) = S$. A \emph{mutation cycle} is a closed walk where no two consecutive vertices (quivers) and edges (mutations) coincide. 
We will generally denote a mutation cycle by some initial quiver $Q$ and a reduced mutation sequence $\mathbf{i}$ such that $\mu_{\mathbf i}(Q) = Q.$
If~$Q$ has no isolated vertices, then any reduced mutation sequence $\mathbf i$ such that $\mu_{\mathbf i}(Q)=Q$ is a mutation cycle.

A mutation cycle is \emph{simple} if it only visits each quiver once; that is, all quivers in $\big [ Q \big ]_{\mathbf i}$ are distinct except for $Q = \mu_{\mathbf i}(Q)$ (cf.\ \cite[Definition~5.1]{LMC}). 

\end{definition}

\begin{definition}
Given a quiver $Q$, we define $2$ additional quivers: the \emph{framed extension} $\widehat{Q}$ and the \emph{coframed extension} $\widecheck{Q}$. 
To construct either extension, we first add a new `frozen' vertex $i'$ for every vertex $i$ in $Q$. 
We then add a new arrow $i \rightarrow i'$ (resp. $i' \rightarrow i$) for each new vertex $i'$ to form $\widehat{Q}$ (resp. $\widecheck{Q}$). See Figure~\ref{fig:3-cycle framing}.

We call the original, non-frozen, vertices \emph{mutable}.
\end{definition}

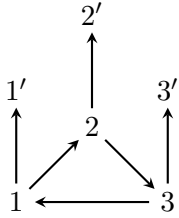
\begin{figure}[ht]
        \centering
        \vspace{-5pt}
        \begin{tikzpicture}
\path (1,0) node {1};
\path (2,1) node {2};
\path (3,0) node {3};

\path (1,1.5) node {$1'$};
\path (2,2.5) node {$2'$};
\path (3,1.5) node {$3'$};

\draw[black, thick, -{stealth}, shorten >=7pt, shorten <= 7pt ] (1,0) -- (1,1.5);
\draw[black, thick, -{stealth}, shorten >=7pt, shorten <= 7pt ] (2,1) -- (2,2.5);
\draw[black, thick, -{stealth}, shorten >=7pt, shorten <= 7pt ] (3,0) -- (3,1.5);

\draw[black, thick, -{stealth}, shorten >=7pt, shorten <= 7pt ] (1,0) to (2,1);
\draw[black, thick, -{stealth}, shorten >=7pt, shorten <= 7pt ] (2,1) to  (3,0);
\draw[black, thick, {stealth}-, shorten >=7pt, shorten <= 7pt ] (1,0) to  (3,0);
\end{tikzpicture}
    \caption{The principle framing of an oriented $3$-cycle.}
    \label{fig:3-cycle framing}
\end{figure}

\begin{definition}
    For a quiver $R$ mutation equivalent to $\widehat Q$, we say a vertex $j$ is \emph{red} (resp. \emph{green}) if for every frozen vertex $i'$ we have $b_{ji'} \leq 0$ (resp. $b_{ji'} \geq 0$).
\end{definition}

\begin{theorem}[Sign-Coherence \cite{MR2629987}]
\label{thm:all red or green}
    For every mutable vertex $j$ in every quiver ${R \in [\widehat Q]}$, $j$ is either red or green.
\end{theorem}

\begin{definition}
     A mutation sequence $\mathbf{M}$ is a \emph{reddening sequence} for $Q$ if every mutable vertex of $\mu_{\mathbf{M}}(\widehat{Q})$ is red.
     It is called a \emph{maximal green} sequence if we only mutate at green vertices.
\end{definition}

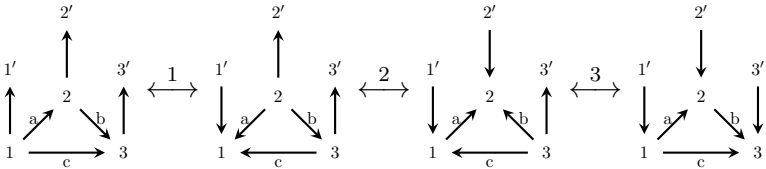
\begin{figure}[ht]
        \vspace{-5pt}
\begin{tikzpicture}[scale=0.75, every node/.style={scale=0.6}]
\path (1,0) node {1};
\path (2,1) node {2};
\path (3,0) node {3};

\path (1,1.5) node {$1'$};
\path (2,2.5) node {$2'$};
\path (3,1.5) node {$3'$};

\draw[black, thick, -{stealth}, shorten >=7pt, shorten <= 7pt ] (1,0) -- (1,1.5);
\draw[black, thick, -{stealth}, shorten >=7pt, shorten <= 7pt ] (2,1) -- (2,2.5);
\draw[black, thick, -{stealth}, shorten >=7pt, shorten <= 7pt ] (3,0) -- (3,1.5);

\draw[black, thick, -{stealth}, shorten >=7pt, shorten <= 7pt ] (1,0) to node [above left, swap, inner sep=0pt] {a}  (2,1);
\draw[black, thick, -{stealth}, shorten >=7pt, shorten <= 7pt ] (2,1) to node [above right, swap, inner sep=0pt] {b}  (3,0);
\draw[black, thick, -{stealth}, shorten >=7pt, shorten <= 7pt ] (1,0) to node [below, swap] {c}  (3,0);
\end{tikzpicture}
$\mathrel{\raisebox{1cm}{$\mutation{1}$}}$
\begin{tikzpicture}[scale=0.75, every node/.style={scale=0.6}]
\path (1,0) node {1};
\path (2,1) node {2};
\path (3,0) node {3};

\path (1,1.5) node {$1'$};
\path (2,2.5) node {$2'$};
\path (3,1.5) node {$3'$};

\draw[black, thick, {stealth}-, shorten >=7pt, shorten <= 7pt ] (1,0) -- (1,1.5);
\draw[black, thick, -{stealth}, shorten >=7pt, shorten <= 7pt ] (2,1) -- (2,2.5);
\draw[black, thick, -{stealth}, shorten >=7pt, shorten <= 7pt ] (3,0) -- (3,1.5);

\draw[black, thick, {stealth}-, shorten >=7pt, shorten <= 7pt ] (1,0) to node [above left, swap, inner sep=0pt] {a}  (2,1);
\draw[black, thick, -{stealth}, shorten >=7pt, shorten <= 7pt ] (2,1) to node [above right, swap, inner sep=0pt] {b}  (3,0);
\draw[black, thick, {stealth}-, shorten >=7pt, shorten <= 7pt ] (1,0) to node [below, swap] {c}  (3,0);
\end{tikzpicture}
$\mathrel{\raisebox{1cm}{$\mutation{2}$}}$
\begin{tikzpicture}[scale=0.75, every node/.style={scale=0.6}]
\path (1,0) node {1};
\path (2,1) node {2};
\path (3,0) node {3};

\path (1,1.5) node {$1'$};
\path (2,2.5) node {$2'$};
\path (3,1.5) node {$3'$};

\draw[black, thick, {stealth}-, shorten >=7pt, shorten <= 7pt ] (1,0) -- (1,1.5);
\draw[black, thick, {stealth}-, shorten >=7pt, shorten <= 7pt ] (2,1) -- (2,2.5);
\draw[black, thick, -{stealth}, shorten >=7pt, shorten <= 7pt ] (3,0) -- (3,1.5);

\draw[black, thick, -{stealth}, shorten >=7pt, shorten <= 7pt ] (1,0) to node [above left, swap, inner sep=0pt] {a}  (2,1);
\draw[black, thick, {stealth}-, shorten >=7pt, shorten <= 7pt ] (2,1) to node [above right, swap, inner sep=0pt] {b}  (3,0);
\draw[black, thick, {stealth}-, shorten >=7pt, shorten <= 7pt ] (1,0) to node [below, swap] {c}  (3,0);
\end{tikzpicture}
$\mathrel{\raisebox{1cm}{$\mutation{3}$}}$
\begin{tikzpicture}[scale=0.75, every node/.style={scale=0.6}]
\path (1,0) node {1};
\path (2,1) node {2};
\path (3,0) node {3};

\path (1,1.5) node {$1'$};
\path (2,2.5) node {$2'$};
\path (3,1.5) node {$3'$};

\draw[black, thick, {stealth}-, shorten >=7pt, shorten <= 7pt ] (1,0) -- (1,1.5);
\draw[black, thick, {stealth}-, shorten >=7pt, shorten <= 7pt ] (2,1) -- (2,2.5);
\draw[black, thick, {stealth}-, shorten >=7pt, shorten <= 7pt ] (3,0) -- (3,1.5);

\draw[black, thick, -{stealth}, shorten >=7pt, shorten <= 7pt ] (1,0) to node [above left, swap, inner sep=0pt] {a}  (2,1);
\draw[black, thick, -{stealth}, shorten >=7pt, shorten <= 7pt ] (2,1) to node [above right, swap, inner sep=0pt] {b}  (3,0);
\draw[black, thick, -{stealth}, shorten >=7pt, shorten <= 7pt ] (1,0) to node [below, swap] {c}  (3,0);
\end{tikzpicture}
\caption{A reddening sequence of a $3$-vertex quiver.}
\label{fig:source seq}
\end{figure}

\begin{definition}
\label{def:acyclic}
A quiver $Q$ is \emph{acyclic} if it is acyclic as a directed graph. That is, there are no oriented cycles in $Q$.
A \emph{source sequence} of $Q$ is a mutation sequence $\mathbf{S} = v_1, \ldots, v_k$ so that $v_i \rightarrow v_j$ implies $ i < j$ for some $k \geq 1$.
A \emph{reddening source sequence} (or \emph{RSS}) is any source sequence which is also a reddening sequence. 
If $Q$ is complete this sequence is unique.
\end{definition}

\begin{example}
\label{eg:source seq}
Let $Q$ be an acyclic quiver. 
If we mutate at a source sequence $\mathbf S$ of $Q$, then the first mutation at each vertex will be a source mutation in $\widehat{Q}$.
Thus any source sequence which mutates at all the vertices of $Q$ once is an RSS of $Q$.
Because each arrow will be reversed twice, we have that $\mu_{\mathbf S}(Q)= Q.$
See Figure~\ref{fig:source seq}. 
\end{example}


\begin{proposition}[{\cite[Proposition 2.10]{MR3250044}}] 
\label{prop:red/green gives iso}
Fix a quiver $Q$. 
Let $R = \mu_{\mathbf{N}}(\widehat Q)$ for some mutation sequence $\mathbf{N}$.
If every mutable vertex of $R$ is green then $R$ is isomorphic to~$\widehat Q$.
If every mutable vertex of $R$ is red then $R$ is isomorphic to~$\widecheck Q$.
Further, in both cases the isomorphism fixes the frozen vertices of $Q$.
\end{proposition}

As an immediate corollary of Proposition~\ref{prop:red/green gives iso}:

\begin{corollary}[\cite{MR3250044}] 
\label{cor:reddening gives iso}
Let $Q$ be a quiver with reddening sequence $\mathbf{N}$.
Then $Q$ is isomorphic to $\mu_{\mathbf{N}}(Q)$.
\end{corollary}

\begin{definition}
Given a quiver $Q$, the associated permutation $\sigma$ to a reddening sequence $\mathbf{S}$ is the unique permutation of the mutable vertices of $Q$ such that:
$$\sigma(\widecheck Q) = \mu_{\mathbf{S}}(\widehat Q).$$ 
\end{definition}

We can construct many examples of reddening sequences (of arbitrarily large length) with the following theorem.

\begin{theorem}[{\cite[Theorem~18]{MR3512669}}] 
\label{thm:muller redd}
If $\mathbf{S}$ is a reddening sequence of a quiver $Q$ with associated permutation $\sigma$ and $\mathbf{M}$ is a sequence of mutations then the mutation sequence $\mathbf{M}^{-1} \mathbf{S} \sigma(\mathbf{M})$ is a reddening sequence of $\mu_{\mathbf{M}}(Q)$.
We call the mutation sequence $\mathbf{M}^{-1} \mathbf{S} \sigma(\mathbf{M})$ the \emph{conjugation} of $\mathbf{S}$ by $\mathbf{M}$
(Note that conjugation depends on the permutation $\sigma$).
\end{theorem}

\begin{example}
\label{eg:red seq}
Let $Q$ be the leftmost quiver on the mutation sequence in Figure~\ref{fig:oneStep}.
By Theorem~\ref{thm:muller redd}, with $\mathbf{M} = 2$ and using the RSS described in Example~\ref{eg:source seq} the mutation sequence $2,1,2,3,2$ is a reddening sequence for $Q$.
We also see that the leftmost and rightmost quivers in Figure~\ref{fig:oneStep} are isomorphic (equal, in fact), in agreement with Corollary~\ref{cor:reddening gives iso}. 
\end{example}

\begin{figure}[ht]
    \centering
    \vspace{-10pt}
    \includegraphics[width=0.95\linewidth]{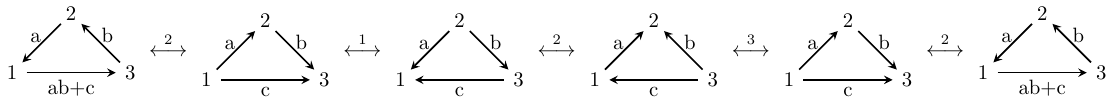}
    \caption{Another reddening sequence of a $3$-vertex quiver.}
    \label{fig:oneStep}
\end{figure}

\begin{definition}[{Cf.\ \cite[Definition 3.7]{MR2640929}, \cite[Theorem 4.4]{MR4186974}}]
\label{def:tri extension}
Given two quivers $H$ and $T$, a \emph{triangular extension} of $H$ and $T$ is another quiver $Q$ formed by taking the disjoint union of $H$ and $T$, and then adding any number of arrows oriented $t \rightarrow h$ for $t \in T,h \in H$.
We form a $|T| \times |H|$ matrix $A = (a_{th})$, where $a_{th}$ is the number of arrows from $t$ to~$h.$ 
We will use the notation $Q = T \stackrel{A}{\rightarrow} H$ for this triangular extension. 
\end{definition}

\begin{remark}
Triangular extensions are also called \emph{direct sums} by some authors (see, for example, \cite{MR4183151, MR3767505, MR3604067}).
The notation $T \stackrel{A}{\rightarrow} H$ for triangular extensions is new.
\end{remark}

\begin{definition}
\label{def:b-matrix}
Let $Q$ be an $n$-vertex quiver.
The associated \emph{$B$-matrix} (or \emph{exchange matrix}) $B(Q)$ is the $n\times n$ skew-symmetric adjacency matrix of $Q$.
Thus the rows and columns of $B(Q)$ are indexed by the vertices of $Q$, and the entries $b_{ij}$ record the number of arrows $i\rightarrow j$ minus the number of arrows $j \rightarrow i$. 
\end{definition}

\begin{example}
The $B$-matrix of a triangular extension $Q = T \stackrel{A}{\rightarrow} H$ is the $|T| + |H|$ block matrix:
$$B(Q) = \begin{pmatrix} B(T) & A \\ -A^T & B(H) \end{pmatrix}.$$
\end{example}

\begin{theorem}[{\cite[Theorem 4.5, Remark 4.6]{MR4029226}}] 
\label{thm:extension red seq}
Suppose that $\mathbf{M}_H$ is a reddening sequence of $H$ and $\mathbf{M}_T$ is a reddening sequence of $T$.
Then any triangular extension $Q = T \stackrel{A}{\rightarrow} H$ has reddening sequence $\mathbf{M}_T\mathbf{M}_H.$
\end{theorem}

Note that Theorem~\ref{thm:extension red seq} puts no constraints on the matrix $A$, except that every entry be non-negative.

The following corollary is a consequence of Corollary~\ref{cor:reddening gives iso} and Theorem~\ref{thm:extension red seq}.

\begin{corollary}
\label{cor:extensions give iso cycles}
Let $H,T$ be quivers with respective reddening sequences $\mathbf{M}_H, \mathbf{M}_T$. 
Then the triangular extension $Q = T \stackrel{A}\rightarrow H$ is isomorphic to $\mu_{\mathbf{M}_T\mathbf{M}_H}(Q)$.
\end{corollary}

\begin{remark}
To the best of our knowledge, Corollary~\ref{cor:extensions give iso cycles} has not appeared in the literature previously. 
However, it is an easy corollary of multiple existing results.
See, in particular, \cite[Theorem~1.1]{MR4183151}.
\end{remark}

\begin{definition}
\label{def:C-mats}
Let $Q$ be a quiver and $\mathbf{M}$ be any mutation sequence.
Then the \emph{$C$-matrix} $C_{\mathbf{M}} = (c_{ij})$ is the $|Q| \times |Q|$ adjacency matrix such that $c_{ij}$ is the number of arrows from the mutable vertex $i$ to the frozen vertex $j'$ in $\mu_{\mathbf{M}}(\widehat{Q})$.
\end{definition}

\begin{remark}
Note that our convention for the $C$-matrix differs from \cite{MR4029226} (but agrees with the conventions of \cite{ervin2024unrestrictedredsizesigncoherence, MR4537352}). 
\end{remark}

\begin{example}
The four quivers depicted in Figure~\ref{fig:source seq} have $C$-matrices:
$$\begin{pmatrix} 1 & 0 & 0 \\ 0 & 1 & 0 \\ 0 & 0 & 1\end{pmatrix}, \begin{pmatrix} -1 & 0 & 0 \\ 0 & 1 & 0 \\ 0 & 0 & 1\end{pmatrix}, \begin{pmatrix} -1 & 0 & 0 \\ 0 & -1 & 0 \\ 0 & 0 & 1\end{pmatrix}, \begin{pmatrix} -1 & 0 & 0 \\ 0 & -1 & 0 \\ 0 & 0 & -1\end{pmatrix}.$$
\end{example}

\begin{theorem}[\cite{MR4092846}] 
\label{thm:C matrices determine}
Let ${\mathbf{M}}, {\mathbf{N}}$ be mutation sequences for a quiver~$Q$.
If $C_{\mathbf{M}} = C_{\mathbf{N}}$ then $\mu_{\mathbf{M}}(\widehat Q) = \mu_{\mathbf{N}}(\widehat Q)$.
\end{theorem}

That is, $C$-matrices determine the quiver (after choosing an initial framed quiver).

\begin{theorem}[{\cite[Lemma 3.1, Theorem 3.2]{MR4029226}}] 
\label{thm:extensions linear comb of framing}
Let $Q = T \stackrel{A}{\rightarrow} H$ be a triangular extension of two quivers $H,T$.
Let $M$ be a mutation sequence of vertices in $T$.
Then the adjacency matrix between vertices of $T$ and vertices of $H$ in $\mu_{\textbf{M}}(Q)$ is the $|T| \times |H|$ matrix $C_{\textbf{M}}A$. 
\end{theorem}

\begin{remark}
Theorem~\ref{thm:extensions linear comb of framing} follows from \cite[Lemma 3.1, Theorem 3.2]{MR4029226}, by setting ${B_1 = B(H)}$, $B_2 = I$ and $P=A$. Note that $B_2$ is sign-coherent by Theorem~\ref{thm:all red or green}.
\end{remark}

\begin{corollary}
\label{cor:oris follow}
The matrix $C_{\textbf{M}}A$ constructed in Theorem~\ref{thm:extensions linear comb of framing} is row sign-coherent, i.e., each row is either non-negative or non-positive (corresponding to if the associated vertex is red or green in $\mu_M(\widehat H)$).
\end{corollary}

\hide{
\begin{theorem}[\cite{MR2629987}]
\label{thm:sign coherence of G}
Let $G = C^{-T}$ for a $C$-matrix $C$. \cS{Standardize notation; probably add subscripts for mutation seqs and Q.}
Then the columns of $G$ are sign-coherent.
\end{theorem}
}

In a few places, it will be useful to discuss the following kinds of quivers.

\begin{definition}
\label{def:abundant}
A quiver $Q$ is \emph{abundant} if there are at least $2$ arrows between every pair of vertices in $Q$. That is, if $|b_{ij}| \geq 2$ for all $i\neq j$.
\end{definition}

\begin{definition}[{\cite[Definition~2.1]{warkentin2014exchange}}]
\label{def:forks}
An abundant quiver $Q$ is called a \emph{fork} if there is a vertex $r$ (called the \emph{point of return}) such that the induced full subquiver formed by deleting $r$ is acyclic, and for every oriented path $i \rightarrow r \rightarrow j$ we have $$b_{ji} > \max(b_{ir}, b_{rj}).$$

The \emph{forkless} part of a mutation class $\big [Q\big ]$ is the set of quivers in the class which are not forks. 
This forms a connected subgraph of the mutation graph of $Q$.
\end{definition}

\begin{definition}[{\cite[Definition~3.7]{warkentin2014exchange}}]
A quiver $Q$ is called a \emph{pre-fork} if there is a pair of vertices $k\neq k'$ such that 
\begin{itemize}
    \item $j \rightarrow k$ (resp. $k \rightarrow j$) if and only if $j \rightarrow k'$ (resp.\, $k' \rightarrow j$) for $j \not \in \{k, k'\}$, and
    \item the full induced subquivers formed by deleting either $k$ or $k'$ are forks with common point of return $r$.
\end{itemize}
\end{definition}

\begin{definition}
\label{def:key}
An acyclic quiver $Q$ is called a \emph{key} if there are a pair of vertices $k, k'$ such that 
\begin{itemize}
\item $j \rightarrow k$ (resp. $k \rightarrow j$) if and only if $j \rightarrow k'$ (resp.\, $k' \rightarrow j$) for $j \not \in \{k, k'\}$, and 
\item the subquiver formed by deleting $k$ (resp. $k'$) is abundant acyclic.
\end{itemize}
\end{definition}

\begin{remark}
Most forks are characterized by having a \emph{unique descent} and being vortex free (see \cite[Section 6]{LMC}).
One important property of a fork is that all mutations (except possibly the mutation at the point of return) again result in forks.
Every mutation applied to an abundant acyclic quiver will result in either a fork or another abundant acyclic.
Preforks (with keys) generalize these properties \cite{MR4695532}.
\end{remark}

\begin{theorem}[{\cite[Corollary~5.3, c.f.\ Figure~18]{MR4695532}}]
\label{thm:finiteKeys}
Suppose $Q$ is a key with $n \geq 4$ vertices. 
If $b_{kk'} = 0$, then $Q$ is mutation equivalent to $n-1$ keys (including itself), each with a distinct source and sink.
If $b_{kk'} = 1$, then $Q$ is mutation equivalent to $2(n-1)$ keys (including itself), where two distinct keys share the same source and sink if and only if they are isomorphic to each other by the transposition $(k,k')$. 
\end{theorem}

\begin{figure}[ht]
    \centering
    \vspace{-10pt} 
    
        \[\begin{tikzcd}[every arrow/.append style = {-{Stealth}}]
             & 2 & \\
            1 &  & 3 
            \arrow[from=1-2,to=2-1,"3"']
            \arrow[from=2-3,to=1-2,"8"']
            \arrow[from=2-1,to=2-3,"2"']
        \end{tikzcd} \quad \quad
        \begin{tikzcd}[every arrow/.append style = {-{Stealth}}]
            & 4 & \\
            1 & 2 & 3 
            \arrow[from=2-2,to=2-1,"2"]
            \arrow[from=2-2,to=2-3,"4"']
            \arrow[from=2-1,to=1-2,"2"]
            \arrow[from=2-2,to=1-2,"3"]
            \arrow[from=2-3,to=1-2,"4"']
        \end{tikzcd} \quad \quad
        \begin{tikzcd}[every arrow/.append style = {-{Stealth}}]
            & 4 & \\
            1 & 2 & 3 
            \arrow[from=2-2,to=2-1,"2"]
            \arrow[from=2-2,to=2-3,"4"']
            \arrow[from=2-1,to=1-2,"8"]
            \arrow[from=1-2,to=2-2,"3"]
            \arrow[from=2-3,to=1-2,"5"']
        \end{tikzcd} 
        \]
    \caption{A fork, key, and prefork.}
    \label{fig:quiver types}
\end{figure}
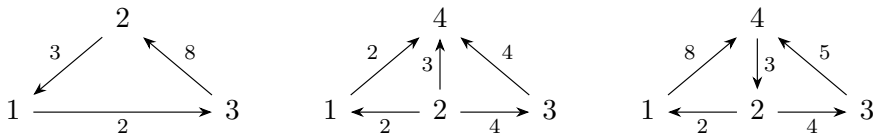

\newpage
\section{Mutation cycles from reddening sequences}
\label{sec:red makes cycles}

We recall our main observation.

\begin{corollary}[{Corollary~\ref{cor:extensions give iso cycles}}]
Let $H,T$ be quivers with respective reddening sequences $\mathbf{M}_H, \mathbf{M}_T$.
Then the triangular extension $Q = T \stackrel{A}\rightarrow H$ is isomorphic to $\mu_{\mathbf{M}_T\mathbf{M}_H}(Q)$.
\end{corollary}

If the reddening sequences each give equal (instead of isomorphic) quivers, then we have a mutation cycle.

\begin{proposition}
\label{prop:extensions with equality give cycles}
    Let $H,T$ be quivers with respective reddening sequences $\mathbf{M}_H, \mathbf{M}_T$, where both sequences have the identity as their associated permutation. 
    If $Q = T \stackrel{A}{\rightarrow} H$, then $\mu_{\mathbf{M}_T\mathbf{M}_H}(Q) = Q$.
\end{proposition}

\begin{proof}
It suffices to show that $\mu_{\mathbf{M}_T}(Q) = H \stackrel{A^T}{\rightarrow} T$; repeating the argument on $\mu_{\mathbf{M}_T}(Q)$ with $H,T$ exchanged gives the claim.

As mutation commutes with restriction, $\mu_{\mathbf{M}_T}(T) = T$ is a subquiver of $\mu_{\mathbf{M}_T}(Q)$. 
By Theorems~\ref{thm:all red or green}~\ref{thm:extensions linear comb of framing}, each mutation $t$ in $\mathbf{M}_T$ is at a red or green vertex in $\big [ T \big ]_{\mathbf M_T}$, and so there are no directed paths from $h \rightarrow t \rightarrow \overline h$ for $h, \overline h \in H$.
Thus each mutation in $\mathbf{M}_T$ leaves the subquiver $H$ unchanged.
Finally, because $\mathbf{M}_T$ is a reddening sequence and Theorem~\ref{thm:extensions linear comb of framing}, $\mu_{\mathbf{M}_T}(\widehat T) = \widecheck T$, and so $\mu_{\mathbf{M}_T}(Q) = H \stackrel{A^T}{\rightarrow} T$.
\end{proof}


This idea can then be extended to reddening sequences that do not produce equality of the mutable subquivers, where the proof is nigh identical to that of Proposition~\ref{prop:extensions with equality give cycles}.

\begin{proposition}
\label{prop:extensions without equality give cycles}
    Let $H,T$ be quivers that admit respective reddening sequences $\mathbf{M}_H, \mathbf{M}_T$, which have respective associated permutations $\sigma$ and $\rho$.  
    Let $Q = T \stackrel{A}{\rightarrow} H$, and let $k>0$ be an integer such that $\sigma^k = \text{id}$ and $\rho^k = \text{id}$. 
    Then $\mu_{\mathbf{S}}(Q) = Q$, where $$\mathbf{S} = \mathbf{M}_T \mathbf{M}_H \rho(\mathbf{M}_T) \sigma(\mathbf{M}_H)  \rho^2(\mathbf{M}_T) \sigma^2(\mathbf{M}_H)  \dots  \rho^{k-1}(\mathbf{M}_T) \sigma^{k-1}(\mathbf{M}_H).$$ 
\end{proposition}

\begin{proof}
    After performing the mutation sequence $\mathbf{M}_T \mathbf{M}_H$, we are left with the triangular extension $  \rho(T)  \stackrel{A}{\rightarrow} \sigma(H) $.
    Repeat this process until equality is achieved. 
\end{proof}

\begin{remark}
By Theorem~\ref{thm:extension red seq}, $\mathbf M_T \mathbf M_H$ is a reddening sequence for $Q$, and from the proof we see that its associated permutation is $\rho \sigma$ (which act on disjoint sets of vertices).
\end{remark}

It is conceivable that the mutation sequence $\mathbf{M}_T \mathbf{M}_H$ may visit the same quiver multiple times.
Theorem~\ref{thm:distinguishing minimal} below shows that this is atypical, but we first require a definition.

\begin{definition}
\label{def:distinguishing}
Fix a quiver $T$ and a mutation sequence $\mathbf{M}$.
Let $I_k$ denote the quiver with $k$ isolated vertices (hence $I_k$ has no arrows). 
A $|T| \times k$ matrix $A$ is \emph{distinguishing} for~$[T]_\mathbf{M}$ if every quiver in $\big [T \stackrel{A}{\rightarrow} I_k \big ]_\mathbf{M}$ is distinct. 
\end{definition}

\begin{remark}
Recall from Theorem~\ref{thm:C matrices determine} that if $C$-matrices agree then so do the quivers themselves.
Thus, if any two $C$-matrices on a mutation sequence agree, then there are no distinguishing matrices for that sequence (and one should instead take a shorter, simple, subsequence).

On the other hand, if all $C$-matrices are distinct then there are always many distinguishing matrices. 
Fix two quivers in $[T]_{\mathbf M}$, with $C$-matrices $C, C'$.
Then the corresponding quivers in $\big [T \stackrel{A}{\rightarrow} I_k\big ]_{\mathbf{M}}$ are distinct whenever $CA \neq C'A$.
Each of these inequality constraints requires that some column vector in $A$ avoids the eigenspace with eigenvalue~$1$ of the matrix $C^{-1} C'$.
Empirically, it seems rare that $C^{-1} C'$ has a positive eigenvector with eigenvalue $1$ at all, and so these constraints are trivial.
For reddening sequences where all $C$-matrices are distinct, we are not aware of a single example of a positive matrix which is not distinguishing. 
Non-distinguishing matrices do exist for arbitrary mutation sequences. 
\end{remark}

\begin{theorem}
\label{thm:distinguishing minimal}
    Let $H,T$ be quivers with respective reddening sequences $\mathbf{M}_H, \mathbf{M}_T$, where both sequences have the identity as their associated permutation. 
    Suppose $A$ is a distinguishing matrix for $\big [T \big ]_{\mathbf{M}_T}$ and $A^T$ is a distinguishing matrix for $\big [H \big ]_{\mathbf{M}_H}$, and that $A$ has a strictly positive row. 
    Then the mutation cycle $\mathbf{M}_T\mathbf{M}_H$, starting at $T \stackrel{A}{\rightarrow} H$, is simple. 
\end{theorem}

\begin{proof}
The mutation sequence $\mathbf{M}_T\mathbf{M}_H$, starting at $T \stackrel{A}{\rightarrow} H$, is a mutation cycle by Theorem~\ref{prop:extensions with equality give cycles}.
So it suffices to show that each pair of quivers on the cycle are distinct.
By Theorem~\ref{thm:C matrices determine} and because $A$ (resp. $A^T$) is distinguishing with respect to $\big [T\big ]_{\mathbf{M}_T}$ (resp. $\big [H\big ]_{\mathbf{M}_H}$), any pair of quivers that are from the first $|\mathbf M_T|$ (resp. last $|\mathbf M_H|$) mutations are distinct.

Suppose finally that $R$ is in $\big [T \stackrel{A}{\rightarrow} H\big ]_{\mathbf{M}_T}$ while $R'$ is in $\big [H \stackrel{A^T}{\rightarrow} T\big ]_{\mathbf{M}_H}$, and that neither is $T \stackrel{A}{\rightarrow} H$ nor $H \stackrel{A^T}{\rightarrow} T$. 
We will differentiate them by showing that there exists vertices $t, \overline t \in T$ and $h \in H$ such that $t \rightarrow h \rightarrow \overline t$ in $R$ but not in $R'$.

By Theorem~\ref{thm:all red or green}, every vertex of $T$ as a subquiver of $R$ is either red or green. 
By Proposition~\ref{prop:red/green gives iso} this subquiver additionally has at least one green and one red vertex (as $R$ is not $T \stackrel{A}{\rightarrow} H$ nor $H \stackrel{A^T}{\rightarrow} T$).
We may set $t$ to any green vertex and $\overline t$ to any red vertex, and set $h$ to the vertex corresponding to the strictly positive row in $A$. 
By Theorem~\ref{thm:extensions linear comb of framing} and Corollary~\ref{cor:oris follow}, there is a path $t \rightarrow h \rightarrow \overline t$ in $R$.
On the other hand, these same results show that every vertex $h \in H$ has only outgoing or incoming arrows to $T$ in $R'$. 
Thus $R \neq R'$.
\end{proof}

\subsection{Reddening Sequences and Forks} \label{subsec:reddening-forks}

As each reddening sequence gives a mutation cycle, we next turn our attention to where reddening sequences can be found in the mutation graph. 
We show that it suffices to restrict to the forkless part (Definition~\ref{def:forks}).

\begin{lemma} \label{lem:reddening forkless part}
    Let $Q$ be a non-fork that admits a reddening sequence.
    Any reduced reddening sequence $\mathbf{N}$ of $Q$ does not pass through a fork.
    In other words, any reddening sequence can be reduced to one that is contained completely in the forkless part of the mutation graph.
\end{lemma}

\begin{proof}
    If $Q$ is mutation-finite or disconnected, then the result follows trivially as $Q$ is not mutation-equivalent to any forks.
    We then assume that $Q$ is mutation-infinite and connected, meaning that it is mutation-equivalent to a fork.
    Let $\mathbf{N}$ be any reduced reddening sequence of $Q$.
    Then $\mu_{\mathbf{N}}(Q)$ is isomorphic to $Q$ by Corollary~\ref{cor:reddening gives iso}.
    Since $Q$ is not a fork, neither is $\mu_{\mathbf{N}}(Q)$ a fork.
    If mutation along $\mathbf{N}$ produces a fork, then we know that further mutation will only produce forks by the Tree Lemma \cite[Lemma~2.8]{warkentin2014exchange}, as we are working with a reduced mutation sequence and can never pass through a point of return.
    Therefore, we cannot arrive at a fork during our reddening sequence, proving the desired result.
\end{proof}

We can then utilize Theorem \ref{thm:muller redd} to show that all reduced reddening sequences behave in such a way: always passing through the forkless part.

\begin{proposition}\label{prop:reddening-forks}
    Let $Q$ be any quiver that admits a reddening sequence.
    Then any reduced reddening sequence $\mathbf{N}$ of $Q$ with associated permutation $\sigma$ is a---possibly trivial---conjugation of a sequence of mutations of non-forks, i.e., $\mathbf{N} = \mathbf{MS\sigma(\mathbf{M}^{-1})}$ for sequences of mutations $\mathbf{M}$ and $ \mathbf{S}$, satisfying the following conditions:
    \begin{itemize}
        \item $\mu_{\mathbf{M}}(Q)$ is not a fork;

        \item every quiver in $[\mu_{\mathbf{M}}(Q)]_\mathbf{S}$ is also not a fork;

        \item the mutation sequence $\mathbf{S}$ is a reddening sequence for $\mu_{\mathbf{M}}(Q)$.
    \end{itemize}
    
\end{proposition} 

\begin{proof}
    If $Q$ is not a fork, then Lemma \ref{lem:reddening forkless part} proves the result immediately.
    Assume then that $Q$ is a fork, and let $\mathbf{M}$ be the shortest sequence of mutations resulting in a non-fork.
    Since $\mathbf{N}$ is reduced, the Tree Lemma \cite{warkentin2014exchange} and Corollary \ref{cor:reddening gives iso} tell us that the reddening sequence must enter the forkless part somewhere along the way.
    Thus $\mathbf{N} = \mathbf{MX}$ for some mutation sequence $\mathbf{X}$, as $\mathbf{M}$ will be the reduced portion entering the forkless part.
    By a similar argument, we know that $\mathbf{N} = \mathbf{MSV}$, where $\mathbf{V}$ is reverse of the shortest mutation sequence taking $\mu_{\mathbf{N}}(Q)$ to the forkless part.
    Let $\sigma$ be the associated permutation of the reddening sequence $\mathbf N$.
    By Corollary \ref{cor:reddening gives iso}, we know that $\mathbf{V} = \sigma(\mathbf{M}^{-1})$.
    The only assertion remaining is to show that $\mathbf{S}$ is a reddening sequence for $\mu_{\mathbf{M}}(Q)$ that remains in the forkless part.
    However, if we conjugate $\mathbf{N}$ by $\mathbf{M}^{-1}$, we get that 
    $$\mathbf{M}^{-1}\mathbf{N} \sigma(\mathbf{M})  = \mathbf{M}^{-1}\mathbf{M S V }\sigma(\mathbf{M}^{-1})^{-1}  = \mathbf{S}.$$
    Thus Theorem \ref{thm:muller redd} tells us that $\mathbf{S}$ must be a reddening sequence for $\mu_{\mathbf{M}}(Q)$. Lastly,  
    Lemma~\ref{lem:reddening forkless part} proves our two assertions on $\mathbf{S}$.
\end{proof}

\begin{remark} \label{rmk:reddening-pre-forks}
    A similar result holds for the case of pre-forkless part.
    This can be used to show that some quivers with a finite pre-forkless part, like Figure \ref{fig:box quiver}, do not admit a reddening sequence.
    
    \begin{figure}[ht]
        \centering
        \[\begin{tikzcd}[every arrow/.append style = {-{Stealth}}]
            1 & 2 \\
            4 & 3
            \arrow[from=1-1,to=1-2, "a"]
            \arrow[from=1-2,to=2-2, "b"]
            \arrow[from=2-2,to=2-1, "a"]
            \arrow[from=2-1,to=1-1, "b"]
        \end{tikzcd}\]
        \caption{Family of quivers with finite pre-forkless part for $a,b \geq 2$.}
        \label{fig:box quiver}
    \end{figure}
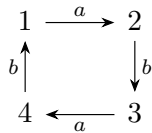
    
    It was additionally shown that each member of this family is mutation-cyclic in a similar manner \cite{warkentin2014exchange}.
\end{remark}

We return to surveying related results, as each quiver with a reddening sequence is now a building block to make mutation cycles.

\begin{lemma} \label{lem:abundant acyclic equality}
    Let $Q$ be an abundant-acyclic quiver on $n \geq 2$ vertices.
    If $\mathbf{N}$ is any reddening sequence, then the associated permutation is the identity.
    Additionally, there is exactly one reduced reddening sequence of $Q$: the reddening source sequence for $Q$. 
\end{lemma}

\begin{proof}
We may assume without loss of generality that $\mathbf{N}$ is reduced.
As $\mathbf{N}$ is reduced and $Q$ is not a fork, then Lemma~\ref{lem:reddening forkless part} shows that $\mathbf{N}$ does not pass through a fork.
Hence, every quiver in $[Q]_\mathbf{N}$ is abundant-acyclic \cite[Lemma 2.15]{MR4695532}, and the sequence~$\mathbf{N}$ 
is a sequence of source or sink mutations.
Add a sink to $Q$ with 2 or more arrows between it and every other vertex to form the quiver $P$.
Then $\mu_{\mathbf{N}}(P)$ is isomorphic to a copy of $P$ where the added vertex has its arrows flipped.
As $\mathbf{N}$ is reduced, we must remain in the forkless part of the mutation class of $P$.
Thus $\mathbf{N}$ is at most $n$ mutations long, as any longer mutation sequence would mutate at the added vertex, which $\mathbf{N}$ does not.
Since any reddening sequence of an acyclic quiver on $n$ vertices must be at least $n$ mutations long, we find that $\mathbf{N}$ must be a sequence of $n$ source mutations.
This is exactly the reddening source sequence for $Q$.
Therefore, we have that $\mu_{\mathbf{N}}(Q) = Q$.
\end{proof}

Combined with Proposition \ref{prop:reddening-forks}, we have an immediate corollary.

\begin{corollary} \label{cor:reddening-forks}
    Let $Q$ be a fork with point of return $r$ that admits a reduced reddening sequence $\mathbf{N} = \mathbf{MS}\sigma(\mathbf{M})^{-1}$ of $Q$ with associated permutation $\sigma$, as in Proposition \ref{prop:reddening-forks}.
    If $P$ is a full subquiver of $Q$, then there are three possibilities:
    \begin{itemize}
        \item If $P$ is abundant acyclic, then the reddening source sequence for $P$ is its only reduced reddening sequence;

        \item If $P$ is a fork and $\mu_{\mathbf{M^*}}(P)$ is abundant acyclic for $\mathbf M^*$ a subsequence of $\mathbf{M}$, then the only reduced reddening sequence for $P$ is $\mathbf{M^* S^*} (\mathbf{ M^*})^{-1}$, where $\mathbf{S^*}$ is the reddening source sequence for $\mu_{\mathbf{M^*}}(P)$;

        \item Otherwise, every reduced reddening sequence for $P$ begins with $\mathbf{M}$.
    \end{itemize}
\end{corollary}

\begin{proof}
    As every subquiver of a fork with point of return $r$ is either abundant acyclic or a fork with point of return $r$, the quiver $P$ is either abundant acyclic or a fork.
    The first two points then follow from Lemma \ref{lem:abundant acyclic equality}, Proposition \ref{prop:reddening-forks}, and Theorem \ref{thm:muller redd}.
    The remaining result comes from Proposition \ref{prop:reddening-forks} and the assumption that $P$ does not reach an abundant acyclic quiver on the path given by $\mathbf{M}$---hence, every quiver on the path is also a fork.
\end{proof}

We can generalize the above result to the case of keys, where $b_{kk'} = 0$.

\begin{lemma} \label{lem:reddening sequences keys 0}
Let $Q$ be a key with vertices $k,k'$ such that $b_{kk'} = 0$ on $n \geq 3$ vertices.
If $\mathbf{N}$ is any reddening sequence, then the associated permutation is the identity.
\end{lemma}

\begin{proof}
    First, form a new quiver $P = I_1 \stackrel{A}{\rightarrow} Q$, where $I_1$ is the quiver with a single vertex and where $A$ has a single column consisting of distinct integers greater than~1 that do not appear as any $b_{ij}$ of $Q$.
    Then $P$ is a key with the same vertices $k, k'$ on $n+1$ total vertices.
    Label the vertex corresponding to $I_1$ in $P$ by $n+1$ and the vertices of $Q$ in $P$ by $1$ through $n$.
    Let $\mathbf{N}$ be any reddening sequence of $Q$ and $\sigma$ its associated permutation.
    By Theorem~\ref{thm:extension red seq}, the quiver $P$ has a reddening sequence formed by first mutating at $n+1$ and then performing $\mathbf{N}$.
    If $\sigma'$ is the associated permutation of this new reddening sequence on $P$, then we naturally have $\sigma(i) = \sigma'(i)$ for $1 \leq i \leq n$.
    By construction and Theorem~\ref{thm:finiteKeys}, there are no other isomorphic quivers in the mutation class of $P$ \cite{MR4695532}.
    Because there are no other isomorphic quivers, by Proposition~\ref{prop:red/green gives iso} we must have $\sigma'(i)=i$ for all $i$.
    Thus $\sigma(i)=i$, and so $\widecheck Q = \mu_\mathbf{N}(\widehat Q)$, naturally forcing $Q = \mu_\mathbf{N}(Q)$
\end{proof}

\begin{corollary}\label{cor:reddening sequences keys 1}
Let $Q$ be a key with vertices $k,k'$ such that $|b_{kk'}| = 1$ on $n \geq 3$ vertices.
If $\mathbf{N}$ is any reddening sequence, then either the associated permutation is the identity or the transposition given by swapping $k$ and $k'$.
\end{corollary} 

\begin{proof}
    Argue as before in Lemma \ref{lem:reddening sequences keys 0}.
    However, by the construction of $P$ and Theorem~\ref{thm:finiteKeys}, there is one other isomorphic quiver in the mutation class of $P$ \cite{MR4695532}, namely $\tau(P)$.
    Hence, Proposition~\ref{prop:red/green gives iso} forces $\sigma'(i) = i$ for all $i \notin \{k,k'\}$.
    Thus $\sigma(i) = i$ for all $i \notin \{k,k'\}$.
    Let $\tau$ be the transposition given by swapping $k$ and $k'$.
    Then either $\widecheck Q = \mu_N(\widehat Q)$ or $\tau(\widecheck Q)  = \mu_N(\widehat Q)$, naturally forcing $Q = \mu_N(Q)$ or $\tau(Q) = \mu_N(Q)$.
\end{proof}

\begin{remark} \label{rmk:reddening sequences keys}
    Note that we do not state that there is exactly one reduced reddening sequence of $Q$ in Lemma \ref{lem:reddening sequences keys 0} as an infinite number of reduced reddening sequences can be constructed. 
    For example, take the quiver $Q$ given in figure~\ref{fig:infinite reduced example}.
    \begin{figure}[ht]
        \centering
        \[\begin{tikzcd}[every arrow/.append style = {-{Stealth}}]
            1 & 2 & 3 \\
            & 4 
            \arrow[from=1-2,to=1-1,"2"']
            \arrow[from=1-2,to=1-3,"2"]
            \arrow[from=2-2,to=1-1,"2"]
            \arrow[from=1-2,to=2-2,"2"]
            \arrow[from=2-2,to=1-3,"2"']
        \end{tikzcd}\]
        \caption{Key $Q$ with vertices $1,3$.} 
        \label{fig:infinite reduced example}
    \end{figure}
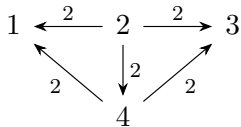
    Let ${\mathbf{N} = 4,1,3,1,3,4,2,4,3,1}$.
    Then both $\mathbf{N}$ and $2,4,3,1$ are easily seen to be reddening sequences.
    Following the structure of the mutation graph of a key additionally shows that an infinite number of reduced reddening sequences may be constructed in this manner \cite{MR4695532}. 
    This occurs as each quiver arrived at in this manner will have no arrows between vertices $1$ and $3$.
    As such, the mutations commute without reducing to a smaller sequence.
\end{remark}

\subsection{Reddening Sequences in Low Rank}

\begin{lemma} \label{lem:reddening sequences rank 2}
    For rank 2 quivers, there are only the following reduced reddening sequences. 
    Let $a$ be the number of arrows in our quiver. 
    Then
    \begin{itemize}
        \item If $a = 0$, then there are two reddening sequences of length two.
        All others can be shortened to a mutation sequence of length two by removing 4-cycles.

        \item If $a = 1$, then there are two reddening sequences: one of length 2 and the other of length 3.
        The first produces a mutable subquiver with equality, and the second produces a mutable subquiver equal up to isomorphism.
        All others can be shortened by removing a portion of the mutation sequence of length 5.

        \item If $a \geq 2$, then there is exactly one reddening sequence of length 2, which produces a mutable subquiver with equality.
        There are no other reddening sequences.
    \end{itemize}
\end{lemma}

\begin{proof}
    The first result follows immediately.
    The last result follows directly from Lemma \ref{lem:abundant acyclic equality}.
    All that remains is the middle result.
    Assume that $Q$ is the quiver given in Figure \ref{fig:framed a2}.
    
    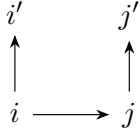
\begin{figure}[ht]
        \centering
        \[\begin{tikzcd}[every arrow/.append style = {-{Stealth}}]
            i' & j' \\
            i & j
            \arrow[from=2-1,to=2-2]
            \arrow[from=2-1,to=1-1]
            \arrow[from=2-2,to=1-2]
        \end{tikzcd}\]
        \caption{A framed copy of $A_2$.}
        \label{fig:framed a2}
    \end{figure}
    
    The reddening sequence of length 2 is just the sequence $i,j$.
    The reddening sequence of length 3 results from the sequence $j,i,j$.
    To see why this is a reddening sequence, note that $j,i,j$ and $i,j,i,j,i,j,i$ both reach the same quiver.
    As $i,j,i,j,i$ produces an isomorphic copy of $Q$, further mutating at $j,i$ performs a reddening sequence for this copy.
    Finally, any reddening sequence of greater length can be shortened to one of the two sequences by removing subsequences of length 5 from the beginning or end of the mutation sequence.
    Each such removal alternates the resulting reddening sequence between equality of the mutable subquiver and an isomorphic copy.
\end{proof}

Naturally, Lemma \ref{lem:reddening sequences rank 2} leads us to a question: does $\big [A_2\big ]$ have to embed into our quiver for a reddening sequence to not produce equality of the mutable subquiver?
This is indeed the case in rank 2, but we can extend it to rank 3 as well.

\begin{lemma} \label{lem:reddening sequences rank 3}
    Let $Q$ be a quiver on 3 vertices.
    If $Q$ admits a reddening sequence and $\big [A_2 \big ]$ does not embed into $\big [ Q \big ]$, then every reddening sequence ends with equality of the mutable subquiver.
\end{lemma}

\begin{proof}
    If $Q$ admits a reddening sequence, then $Q$ must be mutation-acyclic, as $Q$ has only three vertices.
    If $Q$ is additionally mutation-abundant, then any reddening sequence of $Q$ is a conjugation of the reddening sequence given in Lemma \ref{lem:abundant acyclic equality} by Corollary \ref{cor:reddening gives iso}.
    As we are assuming that $\big [A_2\big ]$ does not embed in $\big [ Q \big ]$, we may assume without loss of generality that $Q$ is a key with $b_{kk'}= 0$.
    Then Lemma \ref{lem:reddening sequences keys 0} shows that any reddening sequence ends with equality of the mutable subquiver, proving our result.
\end{proof}

Further evidence comes from abundant acyclic quivers and keys, as shown in Lemmas \ref{lem:abundant acyclic equality} and \ref{lem:reddening sequences keys 0} and Corollary \ref{cor:reddening sequences keys 1}.

\begin{corollary} \label{cor:red-seq-aa-keys}
    Let $Q$ be an abundant acyclic quiver or a key.
    If $P$ is mutation-equivalent to $Q$ and $\big [A_2 \big ]$ does not embed into $\big [ P \big ]$, then every reddening sequence ends with equality of the mutable subquiver.
\end{corollary}


This leads to the following conjecture.

\begin{conjecture}
\label{conj:A2s do all}
For rank 4 and above quivers, a reddening sequence produces the coframed quiver whenever $\big [A_2\big ]$ does not embed into our quiver's mutation class.
\end{conjecture}

\begin{example} \label{eg:dreaded-torus-sec-3}
    The four vertex threshold is where trouble starts to arise.
    Take, for example, the Dreaded Torus given in Figure \ref{fig:dreaded-torus-sec-3}.
    The quiver has many nice properties: it is mutation-finite; it admits a reddening sequence; and it is the only quiver in its mutation-class up to isomorphism.
    Experimentally, we have found that there are many distinct reddening sequences for the Dreaded Torus.
    In particular, we have been able to construct a reddening sequence for every permutation we have tried.
    
    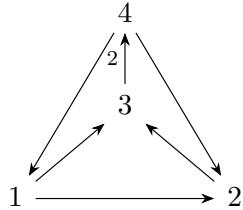
\begin{figure}[ht]
        \centering
        \[\begin{tikzcd}[every arrow/.append style = {-{Stealth}}]
        & 4 \\
        & 3 \\
        1 & & 2
        \arrow[from=2-2, to=1-2, "2"]
        \arrow[from=3-1, to=2-2]
        \arrow[from=3-3, to=2-2]
        \arrow[from=1-2, to=3-1]
        \arrow[from=3-1, to=3-3]
        \arrow[from=1-2, to=3-3]
        \end{tikzcd}\]
        \caption{The Dreaded Torus (the quiver associated to the triangulation of a once punctured torus).}
        \label{fig:dreaded-torus-sec-3}
    \end{figure}
    
\end{example}

\hide{
\begin{remark}
More generally, one could ask which sets of permutations are associated to the reddening sequences of a given quiver. 
Our conjecture is equivalent to saying that if $\big [ A_2 \big ]$ does not embed then this set is trivial.
As shown in Example~\ref{fig:dreaded torus}, this set can be quite large. \cS{Should check: does the Dreaded Torus have the entirety of $S_4$?}
Surprisingly, in every example this set of permutations seems to form a group (if non-empty). 
\end{remark}
\cS{Is the identity always in this set? Is it a group?}
} 

\newpage

\section{Examples} 
\label{sec:examples}

Building on Example~\ref{eg:dreaded-torus-sec-3}, we briefly describe many quivers with reddening sequences from the literature, and give concrete examples.
We then illustrate Propositions~\ref{prop:extensions with equality give cycles}, ~\ref{prop:extensions without equality give cycles} and Corollary~\ref{cor:extensions give iso cycles} by combining our various examples.

Most of these examples come from a broad family of quivers with reddening sequences. In each case, we first give an explicit example of a quiver from the family, a reddening sequence (or sequences), and the associated permutation(s). 
We then remark on relevant terminology or properties of the quiver, and relate these to where else the reddening sequence has appeared in the literature.
Most of this terminology we do not define rigorously, details may be found in the references.
We will restate some theorem(s) regarding this class of reddening sequences before moving on to the next example.

\begin{example}
\label{eg:keyRed}
Let $K$ be the $3$-vertex quiver shown in Figure \ref{fig:first example}.
\begin{figure}[ht]
    \centering
    \[\begin{tikzcd}[every arrow/.append style = {-{Stealth}}]
         & 2 & \\
        1 &  & 3 
        \arrow[from=2-1,to=1-2,"35"]
        \arrow[from=1-2,to=2-3,"4"]
        \arrow[from=2-3,to=2-1,"9"]
    \end{tikzcd}\]
    \caption{Quiver mutation-equivalent to a key.}
    \label{fig:first example}
\end{figure}
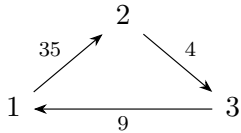

Then $K$ is mutation equivalent to a key $K'$ with arrows $1 \rightarrow 2 \stackrel{4}{\rightarrow} 3 \stackrel{5}{\leftarrow} 1$ (therefore $k,k' = 1,2$).
Specifically, $ K = \mu_{2,3}(K').$
The quiver $K$ has reddening sequences $\mathbf{M} = 3,2,1,2,3,2,3$ and $\mathbf{M}' = 3,2,1,2,3,1,2,1,2,3$, with associated permutations the identity and the transposition $\tau = (1,2)$ respectively.
\end{example}

\begin{remark}
Quivers that are mutation equivalent to an acyclic quiver (such as~$K$) are called \emph{mutation-acyclic}.
Because the RSS of an acylic quiver is a reddening sequence (Example~\ref{eg:source seq}), all mutation-acyclic quivers have a reddening sequence.
\end{remark}

\begin{corollary}[{Theorem~\ref{thm:muller redd}}]
\label{cor:acyclicRed}
Let $H$ be an acyclic quiver with RSS~$\mathbf S$.
Then for any mutation sequence $\mathbf M$, the reduction of the sequence $\mathbf M^{-1} \mathbf S \mathbf M$ is a reddening sequence (with associated permutation the identity) of $\mu_{\mathbf M}(H).$ 
\end{corollary}

\begin{corollary}[{Theorem~\ref{thm:muller redd}, Corollary~\ref{cor:reddening sequences keys 1}}]
\label{cor:key reddening}
Let $H$ be a key with vertices $k,k'$ and $b_{kk'}=1$.
Let $\mathbf S$ be an RSS for $H$, and $\mathbf M$ be any mutation sequence.
Then the reductions of the sequences $\mathbf M^{-1} \mathbf S \mathbf M$ and $\mathbf M^{-1} \mathbf S, k, k', k ,k', k, \tau(\mathbf M)$ are both reddening sequences of $\mu_{\mathbf M}(H)$, with associated permutations the identity and the transposition ${\tau=(k, k')}$ respectively.
\end{corollary} 

We note that the sequence $\mathbf M'$ in Example~\ref{eg:keyRed} is the reduction of $$3,2,1,2,3,\underline{1,2,1,2,1},1,3,$$ where the underlined vertices are $k, k', k, k', k$.

\begin{example}[{\cite[Section~5]{MR4183151}}]
\label{eg:half finite}
Consider the $12$ vertex quiver $Q$ shown in Figure~\ref{fig:half finite recurrent etc}.
Let $\mathbf S_\bullet = 1, 3, 5, 7, 9, 11$ and $\mathbf S_\circ = 2, 4, 6, 8, 10, 12$. 
Then $\mathbf S_\circ \mathbf S_\bullet \mathbf S_\circ \mathbf S_\bullet$ is a reddening sequence for $Q$ with associated permutation 
${\sigma=(1,3) (4,6) (7,9) (10,12)}$.

Note that the mutations at vertices in $\mathbf S_\circ$ (resp. $\mathbf S_\bullet$) all commute, so they may be rearranged and still give a reddening sequence.
\end{example}

\begin{figure}[ht]
\centering
\includegraphics[width=9cm]{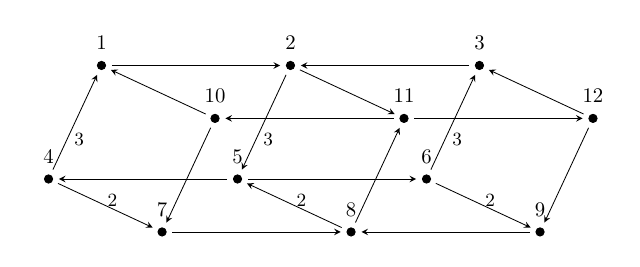}
\caption{
A half-finite bipartite recurrent quiver on $12$ vertices.}
\label{fig:half finite recurrent etc}
\end{figure}

\begin{remark}
\label{rem:half finite notation} 
In the notation of \cite{MR4183151}, the quiver $Q$ discussed in Example~\ref{eg:half finite} would be called \emph{bipartite}, because it is bipartite as an undirected graph.
Color all of the odd indexed vertices~$\bullet$ and the even indexed vertices~$\circ$.
The (non-induced) subgraph which has only the arrows $\bullet \rightarrow \circ$ is a disjoint union of copies of orientations of the Dynkin diagram $A_3$.
Thus this quiver would also be called 
\emph{half-finite}.
(We note that we mutate at $\mathbf S_\circ$~and~$\mathbf S_\bullet$ a total of $4$ times, and the \emph{Coxeter number} of $A_3$ is $4$.)
Further, $Q$ is \emph{recurrent}, meaning that both $\mu_{\mathbf S_\circ}(Q), \mu_{\mathbf S_\bullet}(Q)$ are the \emph{opposite} quiver of $Q$ (that is, $Q$ with all arrows reversed).
\end{remark}

\begin{theorem}[{\cite[Theorem~5.3]{MR4183151}}]
\label{thm:half finite red}
Let $Q$ be a \emph{half-finite bipartite recurrent} quiver. 
Then $Q$ has a maximal green sequence (and thus a reddening sequence). 
\end{theorem}

Theorem~\ref{thm:half finite red} is constructive; the maximal green sequence always consists of alternating the sequences $S_\circ$ and $S_\bullet$, with the number of repetitions determined by the Coxeter number of particular subgraphs of $Q$.


\begin{example}[{\cite[Examples~3.6,4.8]{MR4183151}}]
\label{fig:dreaded torus}
    The Dreaded Torus, given in Figure \ref{fig:dreaded-torus-sec-3}, has a maximal green sequence $1,3,4,2,1,3$ \cite{MR4183151}. A brief computation shows that the associated permutation is $(1,4) (2,3)$.
    The same paper shows that quivers of the form given in Figure \ref{fig:dreaded-torus-dominated} all have the same maximal green sequence.
    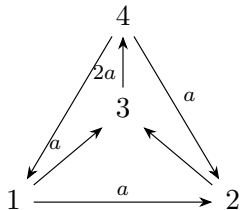
\begin{figure}[ht]
        \centering
        \[\begin{tikzcd}[every arrow/.append style = {-{Stealth}}]
        & 4 \\
        & 3 \\
        1 & & 2 
        \arrow[from=2-2, to=1-2, "2a", pos=0.3]
        \arrow[from=3-1, to=2-2, "a"]
        \arrow[from=3-3, to=2-2]
        \arrow[from=1-2, to=3-1 ]
        \arrow[from=3-1, to=3-3, "a"]
        \arrow[from=1-2, to=3-3, "a"]
        \end{tikzcd}\]
        \caption{Dominated Dreaded Torus (for $a > 0$).}
        \label{fig:dreaded-torus-dominated}
    \end{figure}
\end{example}

\begin{example}[{\cite[Section 5]{MR3695872}}]
\label{eg:punctured spheres}
Fix an integer $k \geq 4$. 
Let $T_k$ be a quiver with $3(k-2)$ vertices, labeled
$v_1, \ldots, v_{k-3}, u_1, \ldots, u_{k-3}, w_{1}, \ldots, w_{k-4}, s, t, \overline s, \overline t$
and arrows 
$$s \rightarrow v_1 \rightarrow v_2 \rightarrow \cdots \rightarrow v_{k-3} \rightarrow t \rightarrow u_{k-3} \rightarrow u_{k-3} \rightarrow \cdots \rightarrow u_{1} \rightarrow s, $$
$$v_1 \rightarrow \overline s \rightarrow u_1, u_{k-3} \rightarrow \overline t \rightarrow v_{k-3},$$
$$ v_{i+1} \rightarrow w_i \rightarrow v_i,$$
$$ u_{i} \rightarrow w_i \rightarrow u_{i+1}.$$ 
We define the mutation sequences
$$\mathbf M_{ind}' = w_1, w_2, \ldots, w_{k-4}, \overline s, \overline t,$$
$$\mathbf M_{cycles} = u_1, v_1, u_2, v_2, \ldots, u_{k-3}, v_{k-3},$$
$$\mathbf M_{ind} = w_1, w_2, \ldots, w_{k-4}, s, t,$$
$$\mathbf M_{X} = v_1, v_2, \ldots, v_{k-3}, \overline t, u_{k-3}, u_{k-4}, \ldots, u_1, \overline s, u_2, u_3, \ldots, u_{k-3}, \overline t, v_{k-3}, v_{k-4}, \ldots, v_1,$$
and their concatenation $\mathbf M = \mathbf M_{ind}' \mathbf M_{cycles} \mathbf M_{ind} \mathbf M_{X}$.
Then $\mathbf M$ is a maximal green sequence of $T_k$ (and, in particular, a reddening sequence) with associated permutation~$( u_1, v_1, \overline s,  s) (t, \overline t)  \prod_{i=2}^{k-3}(v_i, u_i)$.
\end{example}

\begin{remark}
\label{rem:punctured spheres}
The quivers we describe in Example~\ref{eg:punctured spheres} are associated to an~$n \geq 4$ punctured sphere with a specific triangulation. To construct the reddening sequence, we follow the construction of \textcite{MR3695872}. 
Choose a particular puncture~$X$, and denote the other punctures $P_1, \ldots, P_{n-1}$. 
Add compatible edges $P_i - P_{i \pm 1}$ and $P_{n-1} - X - P_{1}$, this divides the sphere into two `hemispheres', each an $n$-gon. 
Now add {two} arcs $X - P_i$ for each $2 \leq i \leq n-2$, one copy in each hemisphere. 
See Figure~\ref{fig:puncturedSphereTriangulation} for the construction with $n=5$ and Example~\ref{eg:punctured sphere 5} for the associated quiver.
\begin{figure}[ht]
\centering
\includegraphics[width=6cm]{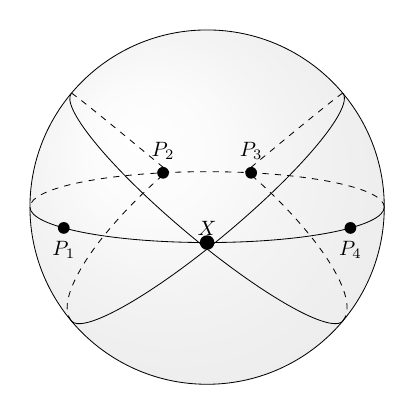}
\caption{A triangulation of a $5$ punctured sphere.}
\label{fig:puncturedSphereTriangulation}
\end{figure}
In the notation of \textcite{MR3695872}, we use the point $X$ as Mills does, thus the sets $\mathcal{M} = \mathcal{M}_0 = \{P_i\},$ and $\mathcal{S} = \varnothing$.
Our mutation sequences $\mathbf M'_{ind}, \mathbf M_{cycles}, \mathbf M_{ind}, \mathbf M_X$ correspond to $\mathscr{M}'_{ind}, \mathscr{M}_{cycles}, \mathscr{M}_{ind}, \mathscr{M}_X$ from \cite[Sections 5.1, 5.2, 5.3, 5.5]{MR3695872} respectively.

A construction similar to Example~\ref{eg:punctured spheres} can be given for other surface quivers. 
A complete description of these quivers and reddening sequences is available in \textcite{MR3695872}.
We do not know a simple or uniform description of their associated permutations.
\end{remark}


\begin{example}
\label{eg:punctured sphere 5}

\begin{figure}[ht]
    \centering
        
    \newcommand{\vertsize}{2pt}
    \begin{tikzpicture}
    \draw[fill] (4,0) circle (\vertsize) node[right=4pt]{$s$} coordinate (s);
    \draw[fill] (-4,0) circle (\vertsize) node[left=4pt]{$t$} coordinate (t);
    \draw[fill] (3,0) circle (\vertsize) node[right=4pt]{$\overline s$} coordinate (sp);
    \draw[fill] (-3,0) circle (\vertsize) node[left=4pt]{$\overline t$} coordinate (tp);

    \draw[fill] (-2,1) circle (\vertsize) node[above=4pt]{$v_2$} coordinate (v2);
    \draw[fill] (2,1) circle (\vertsize) node[above=4pt]{$v_1$} coordinate (v1);
    
    \draw[fill] (-2,-1) circle (\vertsize) node[below=4pt]{$u_2$} coordinate (u2);
    \draw[fill] (2,-1) circle (\vertsize) node[below=4pt]{$u_1$} coordinate (u1);
    
    \draw[fill] (0,0) circle (\vertsize) node[above=4pt]{$w_1$} coordinate (w1);
    
    \draw[-stealth, shorten >=5pt, shorten <= 5pt] (s) -- (v1);
    \draw[-stealth, shorten >=5pt, shorten <= 5pt] (v1) -- (v2);
    \draw[-stealth, shorten >=5pt, shorten <= 5pt] (v2) -- (t);
    \draw[-stealth, shorten >=5pt, shorten <= 5pt] (t) -- (u2);
    \draw[-stealth, shorten >=5pt, shorten <= 5pt] (u2) -- (u1);
    \draw[-stealth, shorten >=5pt, shorten <= 5pt] (u1) -- (s);

    \draw[-stealth, shorten >=5pt, shorten <= 5pt] (v1) -- (sp);
    \draw[-stealth, shorten >=5pt, shorten <= 5pt] (sp) -- (u1);
    
    \draw[-stealth, shorten >=5pt, shorten <= 5pt] (u2) -- (tp);
    \draw[-stealth, shorten >=5pt, shorten <= 5pt] (tp) -- (v2);

    \draw[-stealth, shorten >=5pt, shorten <= 5pt] (v2) -- (w1);
    \draw[-stealth, shorten >=5pt, shorten <= 5pt] (w1) -- (v1);
    \draw[-stealth, shorten >=5pt, shorten <= 5pt] (u1) -- (w1);
    \draw[-stealth, shorten >=5pt, shorten <= 5pt] (w1) -- (u2);
    \end{tikzpicture}
    \caption{The quiver $T_5$.}
    \label{fig:punctured sphere}
\end{figure}
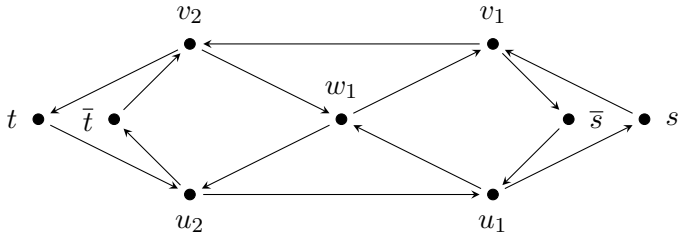

One reddening sequence of $T_5$, as constructed in Example~\ref{eg:punctured spheres} and pictured in Figure~\ref{fig:punctured sphere}, is $$\mathbf S = w_1, \overline s, \overline t,u_1, v_1, u_2, v_2, w_1, s, t, v_1, v_2, \overline t, u_2, u_1, \overline s, u_2, \overline t, v_2, v_1,$$ with associated permutation $\sigma = (u_1, v_1, \overline s, s) (t, \overline t)(u_2, v_2).$

\end{example}

\begin{theorem}[{\cite[Theorem~1.1]{MR3695872}}]
Suppose that $\Sigma$ is a \emph{marked surface} which is not once-punctured and closed. 
Then the quiver associated to any triangulation of $\Sigma$ has a maximal green sequence (and thus a reddening sequence).
\end{theorem}


\begin{example}
\label{eg:banff but tricky}
    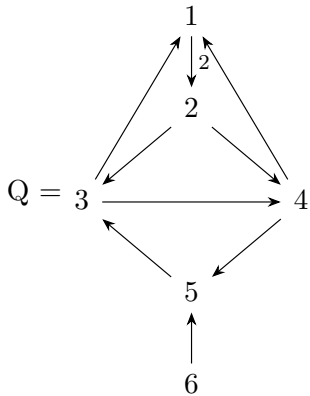
\begin{figure}[ht]
        \centering
        Q =\begin{tikzcd}[every arrow/.append style = {-{Stealth}}]
        & 1 \\
        & 2 \\
        3 & & 4 \\
        & 5\\
        & 6
        \arrow[from=1-2, to=2-2, "2"]
        \arrow[from=2-2, to=3-1]
        \arrow[from=2-2, to=3-3]
        \arrow[from=3-1, to=1-2]
        \arrow[from=3-1, to=3-3]
        \arrow[from=3-3, to=1-2]
        \arrow[from=3-3, to=4-2]
        \arrow[from=4-2, to=3-1]
        \arrow[from=5-2, to=4-2]
        \end{tikzcd}
        \caption{A Banff quiver on 6 vertices.}
        \label{fig:banff example}
    \end{figure}
    All Banff quivers are known to admit a reddening sequence \cite{MR4108215}.
    The quiver $Q$ given in Figure \ref{fig:banff example} is known to be Banff \cite{muller_locally_2013, ervin_answering_2024}.
    After mutating at the sequence 
    $$\mathbf{M} = 2,5,4,1,4,2,1,6,5,4,5,3,$$
    we arrive at a quiver with vertex 4 a source (which one can use for the Banff algorithm).
    This quiver has a reddening sequence given by
    $$\mathbf{S} = 4,1,3,2,3,6,1,5,3,1.$$
    Then $\mathbf{N} = \mathbf{MSM}^{-1}$ is a reddening sequence for our quiver with associated permutation the identity.
\end{example}

\begin{example}
\label{eg:grassmannian reddening}
\begin{figure}[ht]
    \centering
    \begin{tikzcd}[every arrow/.append style = {-{Stealth}}]
    1 & 2 & 3 \\
    4 & 5 & 6 \\
    7 & 8 & 9 \\
    \arrow[from=1-1, to=2-2]
    \arrow[from=1-2, to=2-3]
    \arrow[from=1-2, to=1-1]
    \arrow[from=1-3, to=1-2]
    \arrow[from=2-1, to=3-2]
    \arrow[from=2-1, to=1-1]
    \arrow[from=2-2, to=1-2]
    \arrow[from=2-2, to=2-1]
    \arrow[from=2-2, to=3-3]
    \arrow[from=2-3, to=2-2]
    \arrow[from=2-3, to=1-3]
    \arrow[from=3-1, to=2-1]
    \arrow[from=3-2, to=2-2]
    \arrow[from=3-2, to=3-1]
    \arrow[from=3-3, to=2-3]
    \arrow[from=3-3, to=3-2]
    \end{tikzcd}
    \caption{The triangulated grid quiver $R_{3,3}$.}
    \label{fig:triang grid 3 3}
\end{figure}
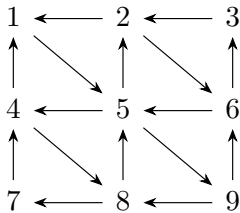
Consider the triangulated grid quiver $R_{3,3}$ shown in Figure \ref{fig:triang grid 3 3}.
This quiver has a reddening sequence $\mathbf S = 7, 4, 1, 8, 7, 5, 4, 2, 1, 9, 8, 7, 6, 5, 4, 3, 2, 1$ \cite{MR3452273} (see Remark~\ref{rem:positroid green seqs}) with associated permutation $\sigma = (1,3) (4,6) (7,9)$.
\end{example}

\begin{example}
\label{eg:positroid reddening}
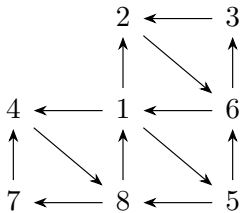
\begin{figure}[ht]
    \centering
    \begin{tikzcd}[every arrow/.append style = {-{Stealth}}]
     & 2 & 3 \\
    4 & 1 & 6 \\
    7 & 8 & 5 \\
    \arrow[from=1-2, to=2-3]
    \arrow[from=1-3, to=1-2]
    \arrow[from=2-1, to=3-2]
    \arrow[from=2-2, to=1-2]
    \arrow[from=2-2, to=2-1]
    \arrow[from=2-2, to=3-3]
    \arrow[from=2-3, to=2-2]
    \arrow[from=2-3, to=1-3]
    \arrow[from=3-1, to=2-1]
    \arrow[from=3-2, to=2-2]
    \arrow[from=3-2, to=3-1]
    \arrow[from=3-3, to=2-3]
    \arrow[from=3-3, to=3-2]
    \end{tikzcd}
    \caption{The quiver $R'$.}
    \label{fig:positroid reddening}
\end{figure}
Consider the quiver $R'$ shown in Figure~\ref{fig:positroid reddening}.
This quiver has a reddening sequence 
$$\mathbf S' = 5, 1, 7, 4, 1, 8, 7, 5, 4, 2, 1, 6, 5, 4, 3, 2, 1, 3, 5,$$ 
with associated permutation $(1,3) (4,6) (7,8)$.
\end{example}

\begin{example}
\label{eg:positroid reddening 2}
\begin{figure}[ht]
    \centering
    \begin{tikzcd}[every arrow/.append style = {-{Stealth}}]
    1 &  & 3 \\
    4 & 5 & 2 \\
    7 & 8 & 9 \\
    \arrow[from=1-1, to=2-3]
    \arrow[from=1-3, to=1-1]
    \arrow[from=2-1, to=3-2]
    \arrow[from=2-1, to=1-1]
    \arrow[from=2-2, to=2-1]
    \arrow[from=2-2, to=3-3]
    \arrow[from=2-3, to=2-2]
    \arrow[from=2-3, to=1-3]
    \arrow[from=3-1, to=2-1]
    \arrow[from=3-2, to=2-2]
    \arrow[from=3-2, to=3-1]
    \arrow[from=3-3, to=2-3]
    \arrow[from=3-3, to=3-2]
    \end{tikzcd}
    \caption{The quiver $R''$.}
    \label{fig:positroid reddening 2}
\end{figure}
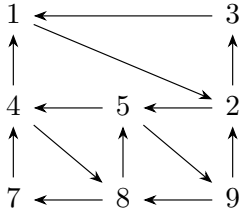
Consider the quiver $R''$ shown in Figure~\ref{fig:positroid reddening 2}.
This quiver has a reddening sequence 
$$\mathbf S'' = 7, 4, 1, 8, 7, 5, 4, 1, 9, 8, 7, 2, 5, 4, 3, 1, 7, 8, 5, 3, 1, 7,$$ 
with associated permutation $ (2,5) (3,8) (4,9,7)$.
\end{example}

\begin{remark} [{\cite{MR3820364}}]
\label{rem:positroid green seqs}
All of the quivers $R_{3,3}, R', R''$ are associated to \emph{reduced plabic graphs}.
Both $\mu_{5,1}(R')$ and $\mu_{2,6}(R'')$ are subquivers of $R_{3,3}$ (by deleting $9$ and $6$ respectively).
It is a theorem that every subquiver of a quiver with a reddening sequence also has a reddening sequence \cite[Theorem~17]{MR3512669}. 
Thus, just from the existence of the reddening sequence $\mathbf S$ of $R$, the quivers $R'$ and $R''$ must have \emph{some} reddening sequence. 
However, it is not clear how to recover reddening sequences for a subquiver from a reddening sequence of the larger quiver (see 
\cite[Remark~35]{MR3512669}).
We found $\mathbf S'$ and $\mathbf S''$ by tedious guesswork.

More generally, let $R_{k\ell}$ be a triangulated grid quiver with side lengths $k$ and $\ell$ respectively (the case $k=\ell=3$ is Example~\ref{eg:grassmannian reddening}).
Then $R_{k\ell}$ has a reddening sequence (as we will describe in Theorem~\ref{thm:grid reddening}).
In \cite[Theorem~4.1]{MR3820364} it was shown that every quiver associated to a reduced plabic graph is mutation equivalent to a subquiver of some $R_{k\ell}$.
Thus by \cite[Theorem~17]{MR3512669}, every quiver associated to a reduced plabic graph has a reddening sequence.
However, there is no systematic description of the reddening sequences nor their associated permutations for all reduced plabic graphs. 
\end{remark}

\begin{theorem}[{\cite[Proposition~11.16]{MR3452273}}]
\label{thm:grid reddening}
Let $R_{k, \ell}$ be the triangulated grid quiver with $k$ rows and $\ell$ columns.
Then $R_{k, \ell}$ has a reddening sequence of length $\binom {\ell+1} 2 k$, constructed as follows:
set $i=1$;
mutate the leftmost $i$ vertices of each row, starting from the bottom row and mutating each row from right-to-left;
increment $i$ and repeat the previous step.
\end{theorem}

\begin{remark}
    The mutation sequence described in Theorem~\ref{thm:grid reddening} is \emph{reversed} from the reddening sequence described in \cite[Section~11]{MR3452273}. 
    This is because we have different conventions for (co)framed extensions.
\end{remark}

\begin{theorem}[{\cite[Theorem~1.2]{MR3820364}}]
\label{thm:reduced plabics redd}
Every quiver associated to a reduced plabic graph has a reddening sequence.
\end{theorem}

\subsection{New mutation cycles}

Our first new construction expands on that of \cite[Theorem~1.1]{LMC}. Combining Proposition~\ref{prop:extensions with equality give cycles} and Corollary~\ref{cor:acyclicRed} immediately gives the following:

\begin{theorem}
\label{thm:LMCGeneral}
    Let $H,T$ be acyclic quivers with RSSs $\mathbf S_H,\mathbf S_T$. 
    Fix mutation sequences $\mathbf N, \mathbf M$ of the vertices in $H, T$ respectively. 
    Then, for any nonnegative $|H|\times |T|$ matrix $A$, the quiver~${Q = \mu_{\mathbf M}(T) \stackrel{A}\rightarrow \mu_{\mathbf N}(H)}$ satisfies $${Q = \mu_{\mathbf M^{-1}\mathbf S_T \mathbf M \mathbf N^{-1} \mathbf S_H \mathbf N}(Q).}$$
    If $Q$ is connected, then the reduction of $\mathbf M^{-1}\mathbf S_T \mathbf M \mathbf N^{-1} \mathbf S_H \mathbf N$ is a mutation cycle.
\end{theorem}

Under mild conditions on the mutation sequences $\mathbf M, \mathbf N$ and quivers $H,T$, we can conclude that this is a simple mutation cycle
(Recall the definition of a distinguishing matrix from Definition~\ref{def:distinguishing}).

\begin{lemma}
\label{lem:PositiveDistinguishingLMC}
Let $T$ be an abundant acyclic quiver 
with RSS $\mathbf S$, and $\mathbf{M}$ a mutation sequence such that $\mathbf M^{-1} \mathbf S \mathbf M$ is reduced.
Then any nonnegative matrix $A \neq 0$ is distinguishing for the mutation sequence $[\mu_{\mathbf M}(T)]_{\mathbf M^{-1} \mathbf S \mathbf M}$.
\end{lemma}

\begin{proof}
Each quiver appears at most twice on the sequence $[\mu_{\mathbf M}(T)]_{\mathbf M^{-1} \mathbf S \mathbf M}$; if we partition $\mathbf M = \mathbf M_1 \mathbf M_2,$ then $\mu_{\mathbf M_1}(T) = \mu_{\mathbf S \mathbf M_1}(T)$, and these are the only equalities. 
Let $I_k$ be the quiver with $k$ isolated vertices.
As 
$$I_k \stackrel{A}{\leftarrow} \mu_{\mathbf M}(T) \neq I_k \stackrel{A^T}{\rightarrow} \mu_{\mathbf M}(T),$$ after applying $\mu_{\mathbf M_2^{-1}}$ to each side we have that 
$$\mu_{\mathbf M_2^{-1}}(I_k \stackrel{A}{\leftarrow} \mu_{\mathbf M}(T)) \neq \mu_{\mathbf M_2^{-1}}(I_k \stackrel{A^T}{\rightarrow} \mu_{\mathbf M}(T)) = \mu_{\mathbf M^{-1} \mathbf S \mathbf M_1}(I_k \stackrel{A}{\leftarrow} \mu_{\mathbf M}(T)).$$
So $A$ is distinguishing for $[\mu_{\mathbf M}(T)]_{\mathbf M^{-1} \mathbf S \mathbf M}$. 
\end{proof}

Combining Lemma~\ref{lem:PositiveDistinguishingLMC} and Theorem~\ref{thm:distinguishing minimal} gives the following:

\begin{theorem}
\label{thm:GeneralizedLMCcycle}
In the setting of Theorem~\ref{thm:LMCGeneral}, if $H,T$ are both abundant acyclic quivers, $\mathbf M^{-1}\mathbf S_T \mathbf M \mathbf N^{-1} \mathbf S_H \mathbf N$ is reduced, and $A \neq 0$ then $\mathbf M^{-1}\mathbf S_T \mathbf M \mathbf N^{-1} \mathbf S_H \mathbf N$ is a simple mutation cycle of $Q = {\mu_{\mathbf M}(T) \stackrel{A}\rightarrow \mu_{\mathbf N}(H)}$.
\end{theorem}

\begin{example}
In Theorem~\ref{thm:LMCGeneral}, take $T = \stackrel{4}{\cdot}$, and $H$ to be the quiver $K'$ in Example~\ref{eg:keyRed}, with $\mathbf N = 2,3$ and $\mathbf M = \emptyset$.
Then the quiver $Q$ in Figure~\ref{fig:keyRed triang} lies on the simple mutation cycles $4,3,2,3,2,1,2,3$ and $4,3,2,1,2,3,1,2,1,2,3$.
\begin{figure}[ht]
    \centering
    \[\begin{tikzcd}[every arrow/.append style = {-{Stealth}}]
     & 2 & \\
     & 4 & \\
    1 &  & 3 
    \arrow[from=3-1,to=1-2,"35"]
    \arrow[from=1-2,to=3-3,"4"]
    \arrow[from=3-3,to=3-1,"9"]
    \arrow[from=2-2,to=3-1,"2"]
    \arrow[from=2-2,to=3-3,"3"']
    \arrow[from=2-2,to=1-2,"4"]
    \end{tikzcd}\]
    \caption{A Triangular Extension of Figure \ref{fig:first example}}
    \label{fig:keyRed triang}
\end{figure}
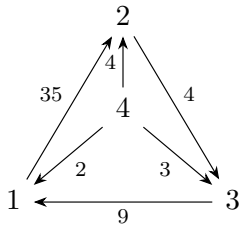
\end{example}

\begin{remark}
In \cite[Section~8]{LMC} the notion of a \emph{fully generic mutation cycle} was introduced.
These are families of $n$-vertex quivers which are parameterized by $\binom n 2$ positive integers, along with a mutation sequence $\mathbf i$ such that $\mathbf i$ is a mutation cycle for all quivers in the family.

Theorem~\ref{thm:GeneralizedLMCcycle} gives many new fully generic mutation cycles. 
By fixing $|H|, |T| > 0$, and the mutation sequences $\mathbf M, \mathbf N, \mathbf S_H, \mathbf S_T$ so that $\mathbf S_H, \mathbf S_T$ are permutations of the vertices in $H, T$, we get a fully generic mutation cycle with $\mathbf i$ equal to the reduction of $\mathbf M^{-1}\mathbf S_T \mathbf M \mathbf N^{-1} \mathbf S_H \mathbf N$ and our family of quivers $Q = {\mu_{\mathbf M}(T) \stackrel{A}\rightarrow \mu_{\mathbf N}(H)}$ parameterized by the choice of weights ($\geq 2$) in $H, T$ and entries of $A$. (The orientations of arrows in $H,T$ are chosen so that $S_H, S_T$ are RSSs.) 
\end{remark}

\begin{remark}
\cite[Theorem~1.1]{LMC} also shows that the mutation cycles constructed there cannot be paved by short cycles.
While we are not aware of any examples where the mutation cycles in Theorem~\ref{thm:GeneralizedLMCcycle} are paved by short cycles, the same proof techniques from \cite{LMC} do not immediately apply, except in the case where $|H|=3$ and $|T|=1$.
We sketch the proof in this case, relying heavily on results from \cite{LMC}.
\end{remark}

\begin{theorem}
\label{thm:4vertNoShortcut}
Let $H$ be an abundant acyclic quiver with RSS~$\mathbf S$ and $|H|=3$.
Let $A$ be a $3\times 1$ matrix with entries $\geq 2$.
Let $\stackrel{4}{\cdot}$ denote the quiver whose only vertex is labeled $4$. 
Then for any mutation sequence $\mathbf M$ such that $\mathbf M^{-1} \mathbf S \mathbf M$ is reduced, the mutation cycle $[\stackrel{4}{\cdot} \stackrel{A}{\rightarrow} \mu_{\mathbf M}(H)]_{4\mathbf M^{-1} \mathbf S \mathbf M}$ is the unique simple mutation cycle in the mutation class $[\stackrel{4}{\cdot} \stackrel{A}{\rightarrow} \mu_{\mathbf M}(H)]$.
\end{theorem}

\begin{lemma}
\label{lem:arrowsGrow}
Every quiver in the mutation cycle $[\stackrel{4}{\cdot} \stackrel{A}{\rightarrow} \mu_{\mathbf M}(H)]_{4 \mathbf M^{-1} \mathbf S \mathbf M}$ is abundant.
\end{lemma}

\begin{proof}
The subquiver supported by the vertices in $H$ in every quiver along the mutation cycle is abundant.
It remains to show that the weights adjacent to vertex $4$ are always $\geq 2$.
We argue by induction on the distance (on the mutation cycle) to $\stackrel{4}{\cdot} \stackrel{A}{\rightarrow} \mu_{\mathbf M}(H)$ or $\stackrel{4}{\cdot} \stackrel{\hspace{3pt} A^T}{\leftarrow} \mu_{\mathbf M}(H)$.
By construction, $\stackrel{4}{\cdot} \stackrel{A}{\rightarrow} \mu_{\mathbf M}(H)$ is abundant. As a source mutation does not change the weights, $\stackrel{4}{\cdot} \stackrel{\hspace{3pt} A^T}{\leftarrow} \mu_{\mathbf M}(H)$ is abundant. This completes the base case.

Suppose the claim is true for all quivers of distance $k$ from $\stackrel{4}{\cdot} \stackrel{A}{\rightarrow} \mu_{\mathbf M}(H)$ or~${\stackrel{4}{\cdot} \stackrel{\hspace{3pt} A^T}{\leftarrow} \mu_{\mathbf M}(H)}$. Suppose $k < |\mathbf M|$.
Then without loss of generality, say $Q'\in [\stackrel{4}{\cdot} \stackrel{A}{\rightarrow} \mu_{\mathbf M}(H)]_{4 \mathbf M^{-1} \mathbf S \mathbf M}$ is distance $k+1$ from $\stackrel{4}{\cdot} \stackrel{A}{\rightarrow} \mu_{\mathbf M}(H)$, and $\mu_i(Q')$ is distance $k$ from $\stackrel{4}{\cdot} \stackrel{A}{\rightarrow} \mu_{\mathbf M}(H)$.
Since $i$ is not an \emph{exit} (\cite[Definition~6.10]{LMC}) of $\mu_i(Q')$ but is a descent of the subquiver supported by the vertices in $H$, it must be an \emph{ascent}(\cite[Definition~3.3]{LMC}) in some $3$-vertex subquiver of $\mu_i(Q')$ involving vertex $4$.
Note that vertex~$4$ is in at most one $3$-cycle.
Thus of the three weights in $Q'$ adjacent to vertex $4$, one increased and two did not change from the corresponding weights of $\mu_i(Q')$. 
In particular, $Q'$ is abundant and every $3$-vertex subquiver involving $i$ and $4$ is either acyclic or has a \emph{descent} (\cite[Definition~3.3]{LMC}) at~$i$.

Now suppose that $k=|\mathbf M|$. Note that all quivers are at most $|\mathbf M|+1$ mutations from $\stackrel{4}{\cdot} \stackrel{A}{\rightarrow} \mu_{\mathbf M}(H)$ or $\stackrel{4}{\cdot} \stackrel{\hspace{3pt} A^T}{\leftarrow} \mu_{\mathbf M}(H)$.
Suppose without loss of generality that $\mathbf S=3,2,1$.
Because $\mathbf M^{-1} \mathbf S \mathbf M$ is reduced, the first vertex in $\mathbf M$ is thus $2$.
We show that 
$$Q' = \mu_{\mathbf M^{-1}1}(\stackrel{4}{\cdot} \stackrel{A}{\rightarrow} \mu_{\mathbf M}(H)) = \mu_{4\mathbf M^{-1}3,2}(\stackrel{4}{\cdot} \stackrel{A}{\rightarrow} \mu_{\mathbf M}(H))$$
is abundant.
There are two cases.
If $1$ is a sink, then the weights of $Q'$ agree with $\mu_1(Q')$, so it is abundant.
Because the subquiver supported by $1,2,4$ is either acyclic or has a descent at $2$ (by the induction argument when $k < |\mathbf M|$), the weight between $2,4$ is larger in $Q'$ than in $\mu_1(Q')$.
So it remains to check that the weight between $3$ and $4$ is larger than $2$.
For this we must turn our attention to the subquiver supported by $2,3,4$ in 
$$Q'' = \mu_{2,3}(Q') = \mu_{4\mathbf M^{-1}}(\stackrel{4}{\cdot} \stackrel{A}{\rightarrow} \mu_{\mathbf M}(H)) = \mu_{\mathbf M^{-1}}(\stackrel{4}{\cdot} \stackrel{\hspace{3pt} A^T}{\leftarrow} \mu_{\mathbf M}(H)).$$ 
By the induction argument when $k < |\mathbf M|$, it is either abundant acyclic or abundant with descent at $2$. 
In either case, this subquiver is still abundant in $Q'$.
A similar argument shows that $\mu_3(Q'')$ is abundant.
\end{proof}

\begin{proof}[Proof of Theorem~\ref{thm:4vertNoShortcut}]
Every abundant $4$-vertex quiver has (at least) $2$ mutations which are exits, which are not on any mutation cycle.
Thus by Lemma~\ref{lem:arrowsGrow} the only possible sequence of mutations that could give a mutation cycle is $4\mathbf M^{-1} \mathbf S \mathbf M$.
\end{proof}

\begin{example}
\label{eg:half-inf and A3}
Let $Q$ be the $12$-vertex quiver from Example~\ref{eg:half finite}, with $\mathbf S_\bullet = 1,3,5,7,9,11$ and $\mathbf S_\circ = 2,4,6,8,10,12$.
The $15$-vertex quiver $P$ shown in Figure~\ref{fig:half finite triang} is a triangular extension of $Q$ and a quiver of finite type $A_3$.
\begin{figure}[ht]
    \centering
    \newcommand{\vertsize}{2pt}
    \begin{tikzpicture}
    \foreach \i in {1,4,7,10} {
    	\draw[fill] (0,0)++(200-\i*90:1.5cm) circle (\vertsize) node[above=4pt]{$\i$} coordinate (v\i);
    	}
    \foreach \i in {2, 5,8,11} {
      \draw[fill] (3.2,0)++(290-\i*90:1.5cm) circle (\vertsize) node[above=4pt]{$\i$} coordinate (v\i);
      }
      
    \foreach \i in {3,6,9,12} {
    	\draw[fill] (6.4,0)++(20-\i*90:1.5cm) circle (\vertsize) node[above=4pt]{$\i$} coordinate (v\i);
    	}
    
    \draw[fill] (0.8,4) circle (\vertsize) node[above=4pt]{$13$} coordinate (v13);
    \draw[fill] (3.4,3) circle (\vertsize) node[above=4pt]{$14$} coordinate (v14);
    \draw[fill] (7.2,4) circle (\vertsize) node[above=4pt]{$15$} coordinate (v15);
    	
    \draw[-stealth, shorten >=5pt, shorten <= 5pt] (v4) -- (v7) node[pos=0.4, right=3pt]{\small $2$};
    \draw[-stealth, shorten >=5pt, shorten <= 5pt] (v10) -- (v7);
    \draw[-stealth, shorten >=5pt, shorten <= 5pt] (v10) -- (v1);
    \draw[-stealth, shorten >=5pt, shorten <= 5pt] (v4) -- (v1) node[pos=0.35, right]{\small $3$};
    
    \draw[-stealth, shorten >=5pt, shorten <= 5pt] (v2) -- (v5) node[pos=0.65, right]{\small $3$};
    \draw[-stealth, shorten >=5pt, shorten <= 5pt] (v8) -- (v5) node[pos=0.6, right=3pt]{\small $2$};
    \draw[-stealth, shorten >=5pt, shorten <= 5pt] (v2) -- (v11);
    \draw[-stealth, shorten >=5pt, shorten <= 5pt] (v8) -- (v11);
    
    \draw[-stealth, shorten >=5pt, shorten <= 5pt] (v6) -- (v9) node[pos=0.4, right=3pt]{\small $2$};
    \draw[-stealth, shorten >=5pt, shorten <= 5pt] (v12) -- (v3);
    \draw[-stealth, shorten >=5pt, shorten <= 5pt] (v12) -- (v9);
    \draw[-stealth, shorten >=5pt, shorten <= 5pt] (v6) -- (v3) node[pos=0.35, right]{\small $3$};

    \draw[-stealth, shorten >=5pt, shorten <= 5pt] (v1) -- (v2);
    \draw[-stealth, shorten >=5pt, shorten <= 5pt] (v3) -- (v2);
    
    \draw[-stealth, shorten >=5pt, shorten <= 5pt] (v7) -- (v8);
    \draw[-stealth, shorten >=5pt, shorten <= 5pt] (v9) -- (v8);
    
    \draw[-stealth, shorten >=5pt, shorten <= 5pt] (v5) -- (v4);
    \draw[-stealth, shorten >=5pt, shorten <= 5pt] (v5) -- (v6);
    
    \draw[-stealth, shorten >=5pt, shorten <= 5pt] (v11) -- (v10);
    \draw[-stealth, shorten >=5pt, shorten <= 5pt] (v11) -- (v12);
    
    \draw[-stealth, shorten >=5pt, shorten <= 5pt] (v1) -- (v13);\draw[-stealth, shorten >=5pt, shorten <= 5pt] (v2) -- (v14);\draw[-stealth, shorten >=5pt, shorten <= 5pt] (v3) -- (v15);
    
    \draw[-stealth, shorten >=5pt, shorten <= 5pt] (v13) -- (v14);
    \draw[-stealth, shorten >=5pt, shorten <= 5pt] (v14) -- (v15);
    \draw[-stealth, shorten >=5pt, shorten <= 5pt] (v15) -- (v13);
     \end{tikzpicture}
    \caption{A triangular extension of Figure~\ref{fig:half finite recurrent etc}.}
    \label{fig:half finite triang}
\end{figure}
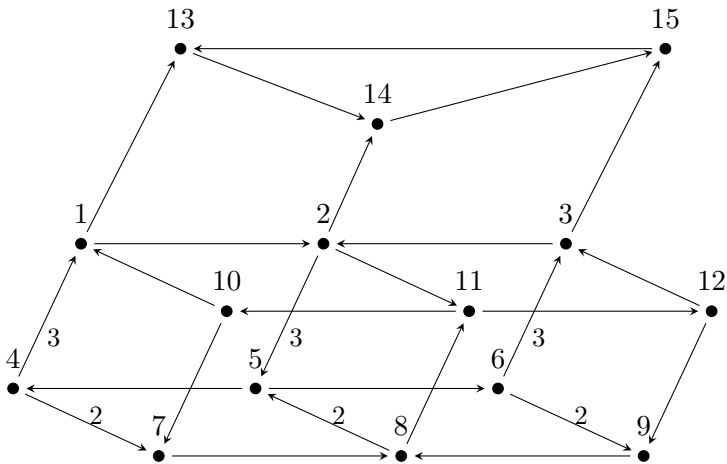

The quiver supported by vertices $\{13,14,15\}$ has several reddening sequences, including $\mathbf M_1 = 14,15,14,13,14$, $\mathbf M_2 = 13,14,15,13$ and $\mathbf M_3 = 13,15,13,14,13$ with respective associated permutations the identity, $\tau = (13,15)$, and $\rho = (13,15,14)$.

By Corollary~\ref{cor:extensions give iso cycles} $P$ is isomorphic to $\mu_{\mathbf S_\circ \mathbf S_\bullet \mathbf S_\circ \mathbf S_\bullet \mathbf M_i}(P)$ for each $i \in \{1,2,3\}$.
Recall that the associated permutation of the reddening sequence $\mathbf S = \mathbf S_\circ \mathbf S_\bullet \mathbf S_\circ \mathbf S_\bullet$ of $Q$ is $\sigma = (1,3) (4,6)(7,9)(10,12).$ 
Thus by Proposition~\ref{prop:extensions without equality give cycles}, we have that
$$P = \mu_{\mathbf S \mathbf M_1 \sigma(\mathbf S) \mathbf M_1}(P),$$
$$P = \mu_{\mathbf S \mathbf M_2 \sigma(\mathbf S) \mathbf \tau(\mathbf M_2)}(P),$$
$$P = \mu_{\mathbf S \mathbf M_3 \sigma(\mathbf S) \rho(\mathbf M_3) \mathbf S \rho^2(\mathbf M_3) \sigma(\mathbf S) \mathbf M_3 \mathbf S \rho(\mathbf M_3) \sigma(\mathbf S) \rho^2(\mathbf M_3)}(P).$$
Direct computation shows that these are simple mutation cycles of lengths $58, 56$ and $174$ respectively.
\end{example}

\begin{example}
\label{eg:S5 and R}
Consider the triangular extension $Q$ of $T_5$ (from Example~\ref{eg:punctured sphere 5})  and $R_{3,3}$ (from Example~\ref{eg:grassmannian reddening}) shown in Figure~\ref{fig:s5 and R}.
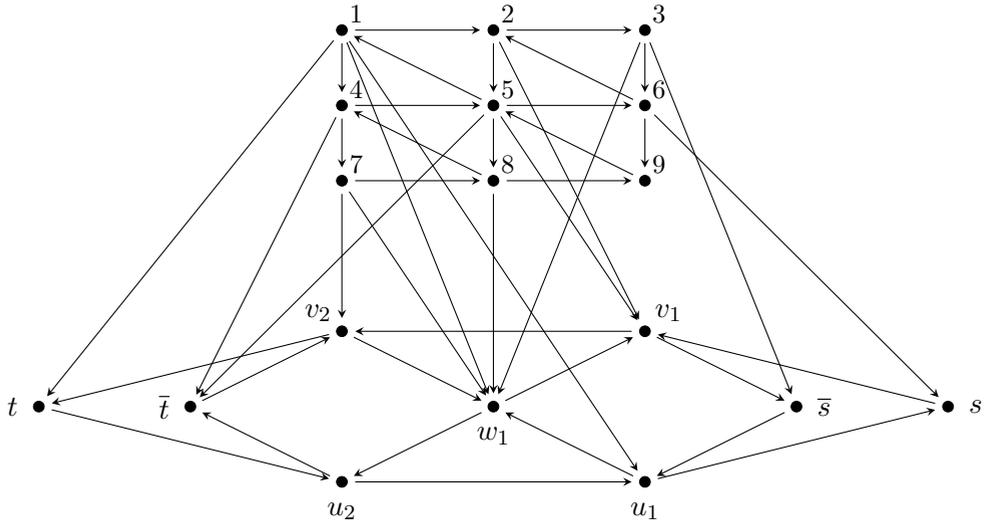
\begin{figure}[ht]
    \centering
    \newcommand{\vertsize}{2pt}
    \begin{tikzpicture}
    \draw[fill] (6,0) circle (\vertsize) node[right=4pt]{$s$} coordinate (s);
    \draw[fill] (-6,0) circle (\vertsize) node[left=4pt]{$t$} coordinate (t);
    \draw[fill] (4,0) circle (\vertsize) node[right=4pt]{$\overline s$} coordinate (sp);
    \draw[fill] (-4,0) circle (\vertsize) node[left=4pt]{$\overline t$} coordinate (tp);

    \draw[fill] (-2,1) circle (\vertsize) node[above left=0pt]{$v_2$} coordinate (v2);
    \draw[fill] (2,1) circle (\vertsize) node[above right=0pt]{$v_1$} coordinate (v1);
    
    \draw[fill] (-2,-1) circle (\vertsize) node[below=4pt]{$u_2$} coordinate (u2);
    \draw[fill] (2,-1) circle (\vertsize) node[below=4pt]{$u_1$} coordinate (u1);
    
    \draw[fill] (0,0) circle (\vertsize) node[below=4pt]{$w_1$} coordinate (w1);
    
    \foreach \i/\j/\n in {-2/5/1, -2/4/4, -2/3/7, 0/5/2, 0/4/5, 0/3/8, 2/5/3, 2/4/6, 2/3/9} {
    \draw[fill] (\i,\j) circle (\vertsize) node[above right=-1pt]{\small $\n$} coordinate (vert\n);
    }
    
    \foreach \i/\j in {1/2, 2/3, 4/5, 5/6, 7/8, 8/9, 1/4, 4/7, 2/5, 5/8, 3/6, 6/9, 5/1, 6/2, 8/4, 9/5} {
    \draw[-stealth, shorten >=5pt, shorten <= 5pt] (vert\i) -- (vert\j);
    }

    \draw[-stealth, shorten >=5pt, shorten <= 5pt] (s) -- (v1);
    \draw[-stealth, shorten >=5pt, shorten <= 5pt] (v1) -- (v2);
    \draw[-stealth, shorten >=5pt, shorten <= 5pt] (v2) -- (t);
    \draw[-stealth, shorten >=5pt, shorten <= 5pt] (t) -- (u2);
    \draw[-stealth, shorten >=5pt, shorten <= 5pt] (u2) -- (u1);
    \draw[-stealth, shorten >=5pt, shorten <= 5pt] (u1) -- (s);

    \draw[-stealth, shorten >=5pt, shorten <= 5pt] (v1) -- (sp);
    \draw[-stealth, shorten >=5pt, shorten <= 5pt] (sp) -- (u1);
    
    \draw[-stealth, shorten >=5pt, shorten <= 5pt] (u2) -- (tp);
    \draw[-stealth, shorten >=5pt, shorten <= 5pt] (tp) -- (v2);

    \draw[-stealth, shorten >=5pt, shorten <= 5pt] (v2) -- (w1);
    \draw[-stealth, shorten >=5pt, shorten <= 5pt] (w1) -- (v1);
    \draw[-stealth, shorten >=5pt, shorten <= 5pt] (u1) -- (w1);
    \draw[-stealth, shorten >=5pt, shorten <= 5pt] (w1) -- (u2);

    \draw[-stealth, shorten >=5pt, shorten <= 5pt] (vert1) -- (t);
    \draw[-stealth, shorten >=5pt, shorten <= 5pt] (vert1) -- (u1);
    \draw[-stealth, shorten >=5pt, shorten <= 5pt] (vert1) -- (w1);
    \draw[-stealth, shorten >=5pt, shorten <= 5pt] (vert2) -- (v1);
    \draw[-stealth, shorten >=5pt, shorten <= 5pt] (vert3) -- (sp);
    \draw[-stealth, shorten >=5pt, shorten <= 5pt] (vert3) -- (w1);
    \draw[-stealth, shorten >=5pt, shorten <= 5pt] (vert4) -- (tp);
    \draw[-stealth, shorten >=5pt, shorten <= 5pt] (vert5) -- (v1);
    \draw[-stealth, shorten >=5pt, shorten <= 5pt] (vert5) -- (tp);
    \draw[-stealth, shorten >=5pt, shorten <= 5pt] (vert6) -- (s);
    \draw[-stealth, shorten >=5pt, shorten <= 5pt] (vert7) -- (v2);
    \draw[-stealth, shorten >=5pt, shorten <= 5pt] (vert7) -- (w1);
    \draw[-stealth, shorten >=5pt, shorten <= 5pt] (vert8) -- (w1);
     \end{tikzpicture}
    \caption{A triangular extension of $T_5$ And $R_{3,3}$.}
    \label{fig:s5 and R}
\end{figure}

Let $\mathbf S$ be the reddening sequence of $T_5$ with associated permutation $\sigma$ (of order $4$) described in Example~\ref{eg:punctured sphere 5} and $\mathbf S'$ be the reddening sequence of $R_{3,3}$ with associated permutation $\sigma'$ (of order $2$) described in Example~\ref{eg:grassmannian reddening}. 
This quiver lies on a simple mutation cycle of length $152$:
$$P = \mu_{\mathbf S' \mathbf S \sigma'(\mathbf S' ) \sigma(\mathbf S ) \mathbf S' \sigma^2(\mathbf S ) \sigma'(\mathbf S') \sigma^3(\mathbf S)}(P).$$ 
\end{example}

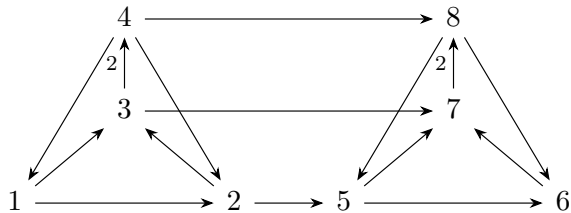
\begin{figure}[ht]
        \centering
    \begin{tikzcd}[every arrow/.append style = {-{Stealth}}]
    & 4 & & & 8\\
    & 3 & & & 7\\
    1 & & 2 & 5 & & 6 \\
    \arrow[from=2-2, to=1-2, "2"]
    \arrow[from=3-1, to=2-2]
    \arrow[from=3-3, to=2-2]
    \arrow[from=1-2, to=3-1]
    \arrow[from=3-1, to=3-3]
    \arrow[from=1-2, to=3-3]
    \arrow[from=2-5, to=1-5, "2"]
    \arrow[from=3-4, to=2-5]
    \arrow[from=3-6, to=2-5]
    \arrow[from=1-5, to=3-4]
    \arrow[from=3-4, to=3-6]
    \arrow[from=1-5, to=3-6]
    \arrow[from=3-3, to=3-4]
    \arrow[from=2-2, to=2-5]
    \arrow[from=1-2, to=1-5]
    \end{tikzcd}
        \caption{Triangular extension of two Dreaded Tori.}
        \label{fig:triangular-dreaded-tori}
    \end{figure}

\begin{example}
    Let $Q$ be the quiver in Figure~\ref{fig:triangular-dreaded-tori}, a triangular extension of two copies of the Dreaded Torus (see Example~\ref{fig:dreaded torus}).
    Then $Q$ has a mutation cycle of length $24$: 
    $$\mathbf S=1,3,4,2,1,3,5,7,8,6,5 ,7,4,2,1,3,4,2,8,6,5,7,8,6.$$
    

    Continuing this theme, the quiver in Figure~\ref{fig:triangular-triangular-dreaded-tori} is a triangular extension of another copy of the Dreaded Torus and the quiver in Figure~\ref{fig:triangular-dreaded-tori}.
    By Theorem~\ref{thm:extension red seq}, the mutation cycle $\mathbf S$ is a reddening sequence of $Q$ (with associated permutation the identity).
    Thus by Proposition~\ref{prop:extensions with equality give cycles}, $\mathbf S, 9,11,12,10,9,11, \mathbf S, 12,10,9,11,12,10$ is a mutation cycle for this quiver. 
\end{example}

\begin{remark}
    By choosing different reddening sequences of the dreaded tori, the quiver in Figure~\ref{fig:triangular-triangular-dreaded-tori} can be shown to lie on several more mutation cycles.
\end{remark}

\begin{figure}[ht]
    \centering
        \centering
    \begin{tikzcd}[every arrow/.append style = {-{Stealth}}]
    & 4 & & & 8 & & & 12\\
    & 3 & & & 7 & & & 11\\
    1 & & 2 & 5 & & 6 & 9 & & 10 \\
    \arrow[from=2-2, to=1-2, "2"]
    \arrow[from=3-1, to=2-2]
    \arrow[from=3-3, to=2-2]
    \arrow[from=1-2, to=3-1]
    \arrow[from=3-1, to=3-3]
    \arrow[from=1-2, to=3-3]
    \arrow[from=2-5, to=1-5, "2"]
    \arrow[from=3-4, to=2-5]
    \arrow[from=3-6, to=2-5]
    \arrow[from=1-5, to=3-4]
    \arrow[from=3-4, to=3-6]
    \arrow[from=1-5, to=3-6]
    \arrow[from=2-8, to=1-8, "2"]
    \arrow[from=3-7, to=2-8]
    \arrow[from=3-9, to=2-8]
    \arrow[from=1-8, to=3-7]
    \arrow[from=3-7, to=3-9]
    \arrow[from=1-8, to=3-9]
    \arrow[from=3-3, to=3-4]
    \arrow[from=2-2, to=2-5]
    \arrow[from=1-2, to=1-5]
    \arrow[from=3-6, to=2-8]
    \arrow[from=3-6, to=3-7]
    \arrow[from=1-5, to=2-8]
    \end{tikzcd}
    \caption{A triangular extension of a triangular extension of two Dreaded Tori and one Dreaded Tori.}
    \label{fig:triangular-triangular-dreaded-tori}
\end{figure}
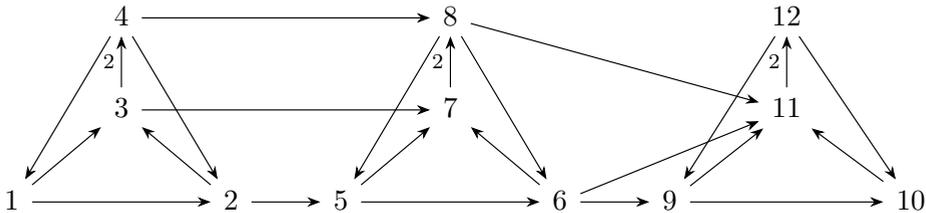

\begin{example}
Consider the triangular extension $R'' \stackrel{A}{\rightarrow} Q$, where $R''$ is as defined in Example~\ref{eg:positroid reddening 2}, $Q$ is from Example~\ref{eg:banff but tricky} with the vertices relabeled $i \mapsto i+9$, and 
$$A = \begin{pmatrix}
0 & 0 & 0 & 0 & 0 &0 \\
0 & 0 & 1 & 0 & 0 &0 \\
0 & 3 & 0 & 0 & 0 &1 \\
0 & 0 & 0 & 0 & 0 &0 \\
0 & 0 & 1 & 0 & 0 &0 \\
0 & 0 & 0 & 0 & 0 &0 \\
0 & 3 & 0 & 0 & 0 &1 \\
0 & 0 & 0 & 0 & 0 &0 \\
\end{pmatrix}.$$

The quiver $R'' \stackrel{A}{\rightarrow} Q$ is shown in Figure~\ref{fig:banff extension example}. 
Recall that $\mathbf S''$ is a reddening sequence of $R''$ with associated permutation $\sigma = (2, 5) (3, 8) (4, 9, 7)$. 
Let $\mathbf N$ be the reddening sequence of $Q$, with the vertices relabeled $i \mapsto i+9$.
Then~$R'' \stackrel{A}{\rightarrow} Q$ lies on the simple mutation cycle of length $336$
$$\mathbf S'' \mathbf N \sigma(\mathbf S'') \mathbf N \sigma^2(\mathbf S'') \mathbf N \sigma^3(\mathbf S'') \mathbf N \sigma^4(\mathbf S'') \mathbf N \sigma^5(\mathbf S'') \mathbf N.$$

\begin{figure}[ht]
    \centering
\begin{tikzcd}[every arrow/.append style = {-{Stealth}}]
    1 &  & 3 & & & & & \\
    4 & 5 & 2 & & & & & \\
    7 & 8 & 9 & & & 12 & & \\
      &   &   & 15 & 14 &   & 11 & 10 \\
      &   &   &   &   & 13 &   & \\
    \arrow[from=1-1, to=2-3]
    \arrow[from=1-3, to=1-1]
    \arrow[from=2-1, to=3-2]
    \arrow[from=2-1, to=1-1]
    \arrow[from=2-2, to=2-1]
    \arrow[from=2-2, to=3-3]
    \arrow[from=2-3, to=2-2]
    \arrow[from=2-3, to=1-3]
    \arrow[from=3-1, to=2-1]
    \arrow[from=3-2, to=2-2]
    \arrow[from=3-2, to=3-1]
    \arrow[from=3-3, to=2-3]
    \arrow[from=3-3, to=3-2]
    \arrow[from=3-2, to=4-4]
    \arrow[from=1-3, to=4-4]
    \arrow[from=3-2, to=4-7, "3", pos=0.4]
    \arrow[from=1-3, to=4-7, "3", pos=0.5, bend left=30]
    \arrow[from=4-4, to=4-5]
    \arrow[from=4-5, to=3-6]
    \arrow[from=4-7, to=3-6]
    \arrow[from=4-7, to=5-6]
    \arrow[from=4-8, to=4-7, "2"']
    \arrow[from=3-6, to=4-8]
    \arrow[from=5-6, to=4-8]
    \arrow[from=3-6, to=5-6]
    \arrow[from=5-6, to=4-5]
    \arrow[from=2-2, to=3-6]
    \arrow[from=2-3, to=3-6]
    \end{tikzcd}
    \caption{A triangular extension of $R''$ and the quiver from Figure~\ref{fig:banff example}. Note this is a $14$-vertex quiver, there is no vertex labeled $6$.}
    \label{fig:banff extension example}
\end{figure}
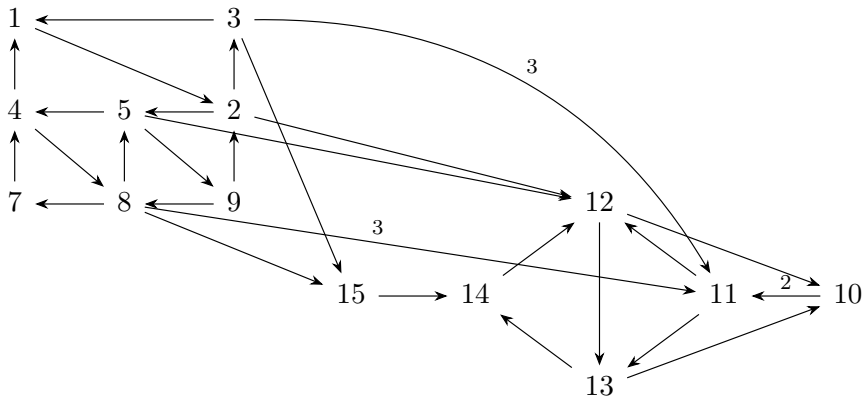

\end{example}

\hide{
\begin{example}
The quiver $A_3 = 1 \leftarrow 2 \leftarrow 3$ lies on several reddening sequences.
For example:
\begin{enumerate}
    \item the RSS $3,2,1$; 
    \item the sequence $3, 1,2,1$; 
    \item the sequence $2, 3, 2, 1$. 
\end{enumerate}
Only the RSS is a mutation cycle; the other two sequences give an isomorphic (but not equal) quiver.
\cS{simulate and visualize}
\end{example}
} 

\subsection{Other cycles}

Not all known mutation cycles are formed from triangular extensions.
We present one new family of mutation cycles (discovered with Sergey Fomin \cite{fominPrivComs}) which generalizes a construction of \textcite{MR2805200}.

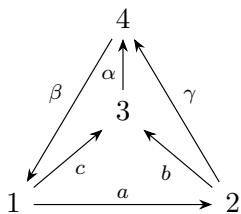
\begin{figure}[ht]
\centering

\[\begin{tikzcd}[every arrow/.append style = {-{Stealth}}]
    & 4 \\
    & 3 \\
    1 & & 2 \\
    \arrow[from=2-2, to=1-2, "\alpha", pos=0.3]
    \arrow[from=3-1, to=2-2, "c", swap]
    \arrow[from=3-3, to=2-2, "b"]
    \arrow[from=1-2, to=3-1, "\beta", swap]
    \arrow[from=3-1, to=3-3, "a"]
    \arrow[from=3-3, to=1-2, "\gamma", swap]
    \end{tikzcd}\]
\caption{When $\alpha = u_{k}(a) - u_{k-2}(a)$, $\beta = u_{k-1}(a)b + u_{k}(a)c$, and $\gamma = u_{k-2}(a)b + u_{k-1}(a)c$ with $a,b,c \geq 2$, this $4$-vertex quiver $Q$ lies on a mutation cycle of length $2k+2$.}
\label{fig:fordy-marsh}
\end{figure}

\begin{example}
\label{eg:fordy-marsh not extension}
To explicitly describe the multiplicities in this example, we will need the \emph{monic Chebyshev Polynomials} $u_k(a)$. 
They are defined recursively, with $u_0(a)=1, u_1(a) = a$ and $u_k(a) = a u_{k-1}(a) - u_{k-2}(a)$ (Cf. \cite[Section 3]{LMC}). 
It is easy to check that $u_k(a) > 0$ whenever $a \geq 2$.

Choose $3$ integers $a,b,c \geq 2$, and an additional positive integer $k$.
Construct a $4$-vertex quiver as shown in Figure~\ref{fig:fordy-marsh}.
Let $\mathbf L$ be the alternating sequence $2,1,2,1, \ldots$ with $|L|=k$.
Then this quiver lies on the simple mutation cycle $\mathbf M = \mathbf L, 4, \sigma(\mathbf L^{-1}), 3.$

Note that $\mu_{\mathbf L, 4}(Q) = \sigma(Q)$, with $\sigma = (1,2) (3,4)$.

Every quiver in $[Q]_{\mathbf M}$ contains an oriented $4$ cycle. 
Thus none of these quivers are a triangular extension of two smaller quivers, and these mutation cycles are not a consequence of Corollary~\ref{cor:extensions give iso cycles}.
Further, no subsequence of $\mathbf M$ is a reddening sequence. 
\end{example}

\hide{
\section{Goals}

\begin{itemize}

    \hide{
    \item Generate all reddening sequences for dreaded torus (see if any are equality; if not, see what prop 2.3 gives. \cS{wrote some code; there are many reddening sequences, but none with length $\leq 10$ which have equality. More than $10$ mutations takes awhile (searching at length $14$ now). Found some, allegedly: [0, 1, 2, 1, 2, 0, 3, 0, 3, 2, 3, 1] , [0, 1, 3, 2, 3, 2, 1, 2, 0, 3, 0, 1] (vertices are for 0 indexed version of Figure~\ref{fig:dreaded torus})}
    \cS{listing a couple shorter reddening seqs; I don't know what automorphism they induce.
    [0, 2, 3, 1, 0, 2]
[0, 1, 2, 1, 3, 2, 0, 1]
[0, 2, 0, 3, 0, 1, 3, 2]}

    \item Add something about the permutation groups you get from reddening sequences. 
    Usually trivial.
    \cS{Note that there's a redd seq which has order $4$, 3134231. When does this occur? Does this reddening seq overlap often? ie, tack on an A2 cycle (before or after) - with enough A2, can get any permutation.}
    \cS{Consider a quiver with 1 arrow between 14 and 23, otherwise abundant; can we say some reddening permutations never occur?}
    }
    
\end{itemize}}

\printbibliography
\end{document}